\documentclass[11pt]{article}
\usepackage{amssymb}
\usepackage{latexsym,amsmath,amssymb,amsfonts,epsfig,graphicx,cite,psfrag}
\usepackage{eepic,color,colordvi,amscd}
\usepackage{mathrsfs}
\usepackage{tikz}
\usepackage{subfigure}
\usepackage[english]{babel}
\usepackage[center]{caption2}
\usepackage{amsthm}
\usepackage{multirow}
\usepackage[usenames,dvipsnames]{pstricks}
\usepackage{epsfig}
\usepackage{pst-grad} 
\usepackage{pst-plot} 
\usepackage[space]{grffile} 
\usepackage{etoolbox} 

\newtheorem{theo}{Theorem}[section]

\newtheorem{lem}[theo]{Lemma}

\newtheorem{conj}[theo]{Conjecture}

\def\qed{\hfill \rule{4pt}{7pt}}

\usepackage{latexsym,amsmath,amssymb,amsfonts,epsfig,graphicx,cite,psfrag}
\usepackage{eepic,color,colordvi,amscd}

\begin{document}
\title{A unified proof of conjectures on cycle lengths in graphs}
\author{Jun Gao\thanks{School of Mathematical Sciences, University of Science and Technology of China, Hefei, Anhui 230026, China. Email: gj0211@mail.ustc.edu.cn.}
\and
Qingyi Huo\thanks{School of Mathematical Sciences, University of Science and Technology of China, Hefei, Anhui 230026, China. Email: qyhuo@mail.ustc.edu.cn.}
\and
Chun-Hung Liu\thanks{Department of Mathematics, Texas A\&M University, College Station, Texas 77843, USA. Email: chliu@math.tamu.edu.
Partially supported by NSF under Grant No. DMS-1664593, DMS-1929851 and DMS-1954054.}
\and
Jie Ma\thanks{School of Mathematical Sciences, University of Science and Technology of China, Hefei, Anhui 230026, China. Email: jiema@ustc.edu.cn.
Partially supported by NSFC grants 11501539 and 11622110, the project ``Analysis and Geometry on Bundles'' of Ministry of Science and Technology of the People's Republic of China, and Anhui Initiative in Quantum Information Technologies grant AHY150200.}
}

\date{}

\maketitle

\begin{abstract}
In this paper, we prove a tight minimum degree condition in general graphs for the existence of paths between two given endpoints,
whose lengths form a long arithmetic progression with common difference one or two.
This allows us to obtain a number of exact and optimal results on cycle lengths in graphs of given minimum degree, connectivity or chromatic number.

More precisely, we prove the following statements by a unified approach.
\begin{enumerate}
\item Every graph $G$ with minimum degree at least $k+1$ contains cycles of all even lengths modulo $k$;
in addition, if $G$ is $2$-connected and non-bipartite, then it contains cycles of all lengths modulo $k$.
\item For all $k\geq 3$, every $k$-connected graph contains a cycle of length zero modulo $k$.
\item Every $3$-connected non-bipartite graph with minimum degree at least $k+1$ contains $k$ cycles of consecutive lengths.
\item Every graph with chromatic number at least $k+2$ contains $k$ cycles of consecutive lengths.
\end{enumerate}
The first statement is a conjecture of Thomassen,
the second is a conjecture of Dean,
the third is a tight answer to a question of Bondy and Vince,
and the fourth is a conjecture of Sudakov and Verstra\"ete.
All of the above results are best possible.
\end{abstract}

\section{Introduction}
The distribution of cycle lengths has been extensively studied in the literature
and remains one of the most active and fundamental research areas in graph theory.
In this paper, along the line of the previous work \cite{LM} of two of the authors,
we investigate various relations between cycle lengths and basic graph parameters such as minimum degree.
The core of the results in \cite{LM} is an optimal bound on the longest sequence of consecutive even cycle lengths in bipartite graphs of given minimum degree.
In the current paper, we extend this result from bipartite graphs to general graphs and use it as a primary tool to derive a number of tight results on cycle lengths in relation to minimum degree, connectivity and chromatic number.
This resolves several conjectures and open problems on cycles of consecutive lengths, cycle lengths modulo a fixed integer and some other related topics.
For a thoughtful introduction on the background, we direct interested readers to \cite{V16,LM}.

Throughout this section, let $k$ be a fixed but arbitrary positive integer, unless otherwise specified.
For a path or a cycle $P$, the {\it length} of $P$, denoted by $|P|$, is the number of edges in $P$.

\subsection{Paths and cycles of consecutive lengths}
The study of cycles of consecutive lengths can be dated back to a conjecture of Erd\H{o}s (see \cite{BV98}) stating that every graph with minimum degree at least three contains two cycles of lengths differing by one or two.
This was solved by Bondy and Vince \cite{BV98} in the following stronger form:
if all but at most two vertices of a graph $G$ have degree at least three,
then $G$ contains two cycles whose lengths differ by one or two.
Since then, this result has inspired extensive research on its generalization to $k$ cycles of consecutive (even or odd) lengths,
including results of H\"{a}ggkvist and Scott \cite{HS98}, Verstra\"ete \cite{V00}, Fan \cite{Fan02}, Sudakov and Verstra\"ete \cite{SV08}, Ma \cite{Ma}, and Liu and Ma \cite{LM}.

We say that $k$ paths or $k$ cycles $P_1,P_2,\ldots,P_k$ are {\it admissible} if $|P_1|\geq 2$
and $|P_1|,|P_2|,\ldots,|P_k|$ form an arithmetic progression of length $k$ with common difference one or two.
The following generalization of Erd\H{o}s' conjecture was posted in \cite{LM},
which was in attempt to attack some related problems.

\begin{conj}[Liu and Ma \cite{LM}]\label{conj:LM}
Every graph with minimum degree at least $k+1$ contains $k$ admissible cycles.
\end{conj}

\noindent By considering the complete graph $K_{k+1}$ or the complete bipartite graph $K_{k,n}$ for any $n\geq k$,
we see that the condition for the minimum degree in Conjecture \ref{conj:LM} is best possible.
Being an evidence, Conjecture \ref{conj:LM} was proved for all bipartite graphs in \cite{LM}.

One of our main results is the following theorem on admissible paths with two given endpoints,
from which Conjecture \ref{conj:LM} can be inferred as a corollary.

\begin{theo}\label{thm 1}
Let $G$ be a $2$-connected graph and let $x,y$ be distinct vertices of $G$.
If every vertex of $G$ other than $x$ and $y$ has degree at least $k+1$, then there exist $k$ admissible paths from $x$ to $y$ in $G$.
\end{theo}

\noindent The case $k=1$ of Theorem \ref{thm 1} is trivial, and the case $k=2$ follows from a result of Fan in \cite{Fan02}.
We remark that Theorem \ref{thm 1} is the main force that will be applied to prove all other results in this paper.

We now show that Conjecture \ref{conj:LM} is an easy corollary of Theorem \ref{thm 1}.
(For possibly ambiguous notations, we refer readers to Section \ref{sec:prelim}.)

\begin{theo}\label{mainthm1}
Every graph $G$ with minimum degree at least $k+1$ contains $k$ admissible cycles.
\end{theo}

\begin{proof}
If $G$ is $2$-connected, let $B=G$ and choose $xy$ to be an arbitrary edge in $G$;
otherwise, let $B$ be an end-block of $G$ with cut-vertex $x$ and choose $y\in N_G(x)\cap V(B-x)$.
Clearly, $B$ is $2$-connected and every vertex of $B$ other than $x$ has degree at least $k+1$.
By Theorem \ref{thm 1}, $B$ contains $k$ admissible paths $P_1,...,P_k$ from $x$ to $y$.
Since each $|P_i|\geq 2$, $P_i\cup xy$ for all $i\in [k]$ form $k$ admissible cycles in $G$.
\end{proof}

Theorem \ref{mainthm1} improves many previous results such as the results in \cite{HS98,V00,Fan02,LM}.
As the writeup of a version of this paper was close to complete, we noticed that very recently, Chiba and Yamashita \cite{CY19} independently proved Theorem \ref{mainthm1} under an extra condition that $G$ is $2$-connected, by using a different approach from this paper.

One can ask another natural question: what are necessary or sufficient conditions for the existence of $k$ cycles of consecutive lengths?
It is clear that such conditions should include non-bipartiteness.
This was addressed by Bondy and Vince in \cite{BV98},
where they proved that any non-bipartite $3$-connected graph contains two cycles of consecutive lengths.
On the other hand, Bondy and Vince showed that the 3-connectivity is necessary by constructing infinitely many non-bipartite $2$-connected graphs
with arbitrarily large minimum degree, yet not containing two cycles of consecutive lengths.

More generally, Bondy and Vince \cite{BV98} asked if there exists a (least) function $f$ such that every non-bipartite $3$-connected graph with minimum degree at least $f(k)$ contains $k$ cycles of consecutive lengths.
The existence of $f(k)$ was confirmed by Fan \cite{Fan02}, where he proved $f(k)\leq 3\lceil k/2\rceil$.
On the other hand, the complete graph $K_{k+1}$ shows $f(k)\geq k+1$.

Our next result determines $f(k)=k+1$ and hence provides the optimal answer to the above question of Bondy and Vince.

\begin{theo}\label{nonbi 3}
Every non-bipartite $3$-connected graph with minimum degree at least $k+1$ contains $k$ cycles of consecutive lengths.
\end{theo}

\subsection{Cycle lengths modulo a fixed integer}
Burr and Erd\H{o}s initiated the study of cycle lengths modulo an integer $k$;
they conjectured (see \cite{Erd76}) that for odd $k$ there exists a constant $c_k$
such that every graph with average degree at least $c_k$ contains cycles of all lengths modulo $k$.
This was proved by Bollob\'as \cite{B77} and then the value $c_k$ was improved to be $O(k^2)$ by Thomassen in \cite{Th83,Th88}.
Thomassen also proposed two conjectures in \cite{Th83} as follows.

\begin{conj}[Thomassen \cite{Th83}]\label{conj:Thom1}
Every graph with minimum degree at least $k+1$ contains cycles of all even lengths modulo $k$.
\end{conj}

\begin{conj}[Thomassen \cite{Th83}]\label{conj:Thom2}
Every $2$-connected non-bipartite graph with minimum degree at least $k+1$ contains cycles of all lengths modulo $k$.
\end{conj}

\noindent We remark that 2-connectivity and non-bipartiteness are necessary
for even $k$ in Conjecture \ref{conj:Thom2}; see \cite{LM} for explanations.
The minimum degree condition in Conjectures \ref{conj:Thom1} and \ref{conj:Thom2} are tight,
since $K_{k+1}$ has no cycle of length 2 modulo $k$, and $K_{k,n}$ has no cycle of length 2 modulo $k$ for $n\geq k$ and odd $k$.

Results of Verstra\"ete \cite{V00}, Fan \cite{Fan02}, Diwan \cite{D10} and Ma \cite{Ma} indicate that the minimum degree at least $O(k)$ suffices for both conjectures.
For fixed $m\geq 3$ and large $k$, Sudakov and Verstra\"ete \cite{SV17} determined the optimal minimum degree condition for cycles of length $m$ modulo $k$ up to a constant factor.

In \cite{LM}, Liu and Ma confirmed both Conjectures \ref{conj:Thom1} and \ref{conj:Thom2} for even $k$.
They also proved that minimum degree $k+4$ suffices for odd $k$,
and observed that an affirmative of Conjecture \ref{conj:LM} would
imply both Conjectures \ref{conj:Thom1} and \ref{conj:Thom2} for odd $k$.
Therefore, as an immediate corollary of Theorem \ref{mainthm1}, we obtain the following.

\begin{theo}\label{thm:twoThomassen}
Conjectures \ref{conj:Thom1} and \ref{conj:Thom2} hold for any positive integer $k$.
\end{theo}

\noindent We would like to address that very recently,
Chiba and Yamashita \cite{CY19} independently proved Conjecture \ref{conj:Thom2}.
Also very recently, Lyngsie and Merker \cite{LM19} proved that for odd $k$,
every $3$-connected cubic graph of large order contains cycles of all lengths modulo $k$.

The case of cycles of length zero modulo $k$ has received considerable attention.
Thomassen \cite{Th88} gave a polynomial-time algorithm for finding a cycle of length zero modulo $k$ in any graph or a certificate that no
such cycle exists.
In 1988, Dean \cite{Dean} proposed the following conjecture.

\begin{conj}[Dean \cite{Dean}]\label{deanconj}
For any positive integer $k\geq 3$, every $k$-connected graph contains a cycle of length zero modulo $k$.
\end{conj}

\noindent We point out that Conjecture \ref{deanconj} is tight, as for all odd $k$ and $n\geq k-1$,
the complete bipartite graph $K_{k-1, n}$ is $(k-1)$-connected but has no cycles of length zero modulo $k$.
The case $k=3$ in Conjecture \ref{deanconj} was proved by Chen and Saito \cite{CS}, and the case $k=4$ was solved by Dean, Lesniak and Saito \cite{DLS93}.
To the best of our knowledge, this conjecture remained open for any $k\geq 5$ prior to this paper.

Taking advantage of Theorem \ref{thm 1}, we are able to resolve Conjecture \ref{deanconj} completely.

\begin{theo}\label{deantheo}
Conjecture \ref{deanconj} holds for any positive integer $k\geq 3$.
\end{theo}
\noindent It turns out that the case $k=5$ is the most difficult case for our approach.
We would like to point out that for $k\geq 6$, in many cases in fact we are able to find $k$ admissible cycles.
In particular, our proofs show that when $k\geq 6$, any $k$-connected graph contains cycles of all even lengths modulo $k$,
except for the residue class $2$ modulo $k$ (see Theorem \ref{thm:Deanbigd} for the precise statement).\footnote{To see the tightness, note that both of $K_{k+1}$ (for even and odd $k$) and $K_{k,n}$ (for odd $k$ and $n\geq k$) are $k$-connected and contain cycles of all lengths $2t$ modulo $k$, except for cycles of lengths in the residue class $2$ modulo $k$.}

\subsection{Consecutive cycle lengths and chromatic number}
There has been extensive research on the relation between the chromatic number and cycle lengths.
For a graph $G$, let $L_e(G)$ and $L_o(G)$ be the set of even and odd cycle lengths in $G$, respectively.
Bollob\'as and Erd\H{o}s conjectured and Gyarf\'as \cite{Gyarfas} proved that $\chi(G)\leq 2|L_o(G)|+2$ for any graph $G$.
Mihok and Schiermeyer \cite{MS04} proved an analog for even cycles that $\chi(G)\leq 2|L_e(G)|+3$ for any graph $G$.
A strengthening of the above result was obtained in \cite{LM},
where the number of even cycle lengths $|L_e(G)|$ was replaced by the longest sequence of consecutive even cycle lengths in $G$.
Confirming a conjecture of Erd\H{o}s \cite{Erd92}, Kostochka, Sudakov and Verstra\"ete \cite{KSV} proved that
every triangle-free graph $G$ with $\chi(G)=k$ contains at least $\Omega(k^2\log k)$ cycles of consecutive lengths.

For $k\geq 2$, let $\chi_k$ be the largest chromatic number of a graph which does not contain $k$ cycles of consecutive lengths.
The complete graph $K_{k+1}$ shows that $\chi_k\geq k+1$.
In \cite{SV17}, Sudakov and Verstra\"ete conjectured that the chromatic number of a graph can be bounded by the longest sequence of consecutive cycle lengths from above.

\begin{conj}[Sudakov and Verstra\"ete \cite{SV17}]\label{SV}
For every integer $k\geq 2$, $\chi_k=k+1$.
\end{conj}

Using Theorem \ref{thm 1}, we are able to prove Conjecture~\ref{SV}.

\begin{theo}\label{chroconcy}
Conjecture~\ref{SV} holds for every integer $k\geq 2$.
\end{theo}

The rest of the paper is organized as follows.
In Section \ref{sec:prelim}, we define the notations and include some preliminaries.
In Section \ref{sec:admiss_path}, we prove Theorem \ref{thm 1}.
In Section \ref{sec:consecutive_cycles}, we prove Theorems \ref{nonbi 3} and \ref{chroconcy} by a unified approach via Theorem \ref{thm 1}.
In Section \ref{sec:Dean_proof}, we prove Theorem \ref{deantheo} by extensively applying Theorem \ref{thm 1} as well.

\section{Preliminaries} \label{sec:prelim}
All graphs in this paper are finite, undirected, and simple.
Let $H$ be a subgraph of a graph $G$.
We say that $H$ and a vertex $v\in V(G)-V(H)$ are {\it adjacent} in $G$ if $v$ is adjacent in $G$ to some vertex in $V(H)$.
Let $N_G(H):=\bigcup_{v\in V(H)}N_G(v)-V(H)$ and $N_G[H]:=N_G(H)\cup V(H)$.
For a subset $S$ of $V(G)$, $G[S]$ denotes the subgraph induced by $S$ in $G$, and $G-S$ denotes the subgraph $G[V(G)-S]$; we say that a vertex $v$ and $S$ are {\it adjacent} in $G$ if $v$ is adjacent in $G$ to some vertex in $S$.
For two distinct vertices $x,y$ of $G$, we define $G+xy$ to be the graph with $V(G+xy)=V(G)$ and $E(G+xy)=E(G) \cup \{xy\}$.
A {\it clique} in $G$ is a subset of $V(G)$ whose vertices are pairwise adjacent in $G$.
A vertex is a {\it leaf} in $G$ if it has degree one in $G$.
We say that a path $P$ is {\it internally disjoint} from $H$ if no vertex of $P$ other than its endpoints is in $V(H)$.
For a positive integer $k$, we write $[k]$ for the set $\{1,2,...,k\}$.

For a graph $G$ and a subset $S$ of $V(G)$, we say that a graph $G'$ {\it is obtained from $G$ by contracting $S$} into a vertex $s$, if $V(G')=(V(G)-S)\cup\{s\}$ and $E(G')=E(G-S)\cup\{vs: v\in V(G)-S$ is adjacent to $S$ in $G\}$.

A vertex $v$ of a graph $G$ is a {\em cut-vertex} of $G$ if $G-v$ contains more components than $G$.
A {\em block} $B$ in $G$ is a maximal connected subgraph of $G$ such that there exists no cut-vertex of $B$.
So a block is an isolated vertex, an edge or a $2$-connected graph.
An {\em end-block} in $G$ is a block in $G$ containing at most one cut-vertex of $G$.
If $D$ is an end-block of $G$ and a vertex $x$ is the only cut-vertex of $G$ with $x \in V(D)$, then we say that $D$ is an {\em end-block with cut-vertex $x$}.
Let ${\mathcal B}(G)$ be the set of blocks in $G$ and ${\mathcal C}(G)$ be the set of cut-vertices of $G$. The {\em block structure} of $G$ is the bipartite graph with bipartition $({\mathcal B}(G),{\mathcal C}(G))$, where $x\in {\mathcal C}(G)$ is adjacent to $B\in {\mathcal B}(G)$ if and only if $x\in V(B)$. Note that the block structure of any graph $G$ is a forest, and it is connected if and only if $G$ is connected.
For notations not defined here, we refer readers to \cite{LM}.

The next result can be derived from a special case of \cite[Theorem 2.5]{Fan02}. 

\begin{theo}\label{Thm:Fan}
Let $G$ be a $2$-connected graph and let $x,y$ be distinct vertices of $G$.
If every vertex in $G$ other than $x$ and $y$ has degree at least $3$,
then there are two admissible paths from $x$ to $y$ in $G$.
\end{theo}

\section{Admissible paths}\label{sec:admiss_path}
In this section, we prove Theorem \ref{thm 1}.
We say that $(G,x,y)$ is a {\it rooted graph} if $G$ is a graph and $x,y$ are two distinct vertices of $G$.
The {\it minimum degree} of a rooted graph $(G,x,y)$ is $\min \{d_G(v):v\in V(G)-\{x,y\}\}$.
We also say that a rooted graph $(G,x,y)$ is {\it $2$-connected} if $G+xy$ is $2$-connected.
Theorem \ref{thm 1} is an immediate corollary of the following theorem.

\begin{theo}\label{mainthm}
Let $k$ be a positive integer.
If $(G,x,y)$ is a $2$-connected rooted graph with minimum degree at least $k+1$, then there exist $k$ admissible paths from $x$ to $y$ in $G$.
\end{theo}

The rest of this section is devoted to a proof of Theorem \ref{mainthm}. 
We need the following lemma.

\begin{lem}\label{lemconcatenat}
Let $(H,u,v)$ be a rooted graph and $W$ be a subset of $V(H)$.
Let $s$ be a positive integer.
Assume that there exist $s$ admissible paths $P_1,...,P_s$, where $P_i$ is from $u$ to some $w_i\in W$ for each $i\in [s]$.
Assume that for each $i\in [s]$, $H-V(P_i-w_i)$ contains $t$ paths $R^{i}_1,..., R^{i}_t$ from $w_i$ to $v$
such that their lengths form an arithmetic progression with common difference one or two\footnote{Here, we allow that some path $R^i_j$ has length one.}.
If $|R^{1}_j|=\cdots=|R^{s}_j|$ for every $j\in [t]$, then there exist $s+t-1$ admissible paths in $H$ from $u$ to $v$.
\end{lem}

\begin{proof}
If each of $A$ and $B$ is an arithmetic progression with common difference one or two, then $A+B=\{a+b: a\in A, b\in B\}$ also forms an arithmetic progression with common difference one or two of size at least $|A|+|B|-1$.
So the set $\{P_i \cup R_j^i: i \in [s], j\in [t]\}$ contains $s+t-1$ admissible paths between $u$ and $v$ in $H$.
\end{proof}

Throughout the rest of this section, let $(G,x,y)$ be a counterexample of Theorem \ref{mainthm} with minimum $|V(G)|+|E(G)|$.
That is, for any $2$-connected rooted graph $(H,u,v)$ with $|V(H)|+|E(H)|<|V(G)|+|E(G)|$,
if the minimum degree of $(H,u,v)$ is at least $\ell+1$, then there exist $\ell$ admissible paths from $u$ to $v$ in $H$.

We now prove a sequence of lemmas and then, according to the order of some specified component (this will be clear after Lemma \ref{casebi}), the remaining proof will be divided into two subsections which we handle separately.

\begin{lem}\label{lem 8}
$G$ is $2$-connected, $x$ and $y$ are not adjacent in $G$, and $k\geq 3$.
\end{lem}

\begin{proof}
Theorem~\ref{mainthm} is obvious when $k=1$, and it follows from Theorem \ref{Thm:Fan} when $k=2$.
So $k\geq 3$.
Note that $|V(G)|\geq 4$, for otherwise, $|V(G)|=3$ and $(G,x,y)$ has minimum degree two and thus $k=1$, a contradiction.

Since $G+xy$ is $2$-connected, $G$ is connected.
Suppose that $G$ is not $2$-connected.
Then there exists a cut-vertex $b$ and two connected subgraphs $G_1,G_2$ of $G$ on at least two vertices such that $G=G_1\cup G_2$ and $V(G_1)\cap V(G_2)=\{b\}$.
We may assume that $x\in V(G_1)-b$, $y\in V(G_2)-b$ and by symmetry, $|V(G_1)|\geq 3$.
Then it is straightforward to see that $(G_1,x,b)$ is $2$-connected and has minimum degree at least $k+1$.
By the minimality of $G$, there exist $k$ admissible paths in $G_1$ from $x$ to $b$.
By concatenating each of these paths with a fixed path in $G_2$ from $b$ to $y$, we obtain $k$ admissible paths in $G$ from $x$ to $y$, a contradiction.
Therefore $G$ is $2$-connected.

Suppose that $x$ is adjacent to $y$ in $G$. Let $G'=G-xy$. Since $G$ is $2$-connected, clearly $(G',x,y)$ is $2$-connected and has minimum degree at least $k+1$.
By the minimality of $G$, $G'$ (and thus $G$) contains $k$ admissible paths from $x$ to $y$, a contradiction.
\end{proof}

\begin{lem}\label{case1}
There is no clique in $G-y$ of size at least three containing $x$, and there is no clique in $G-x$ of size at least three containing $y$.
\end{lem}

\begin{proof}
Suppose to the contrary that there is a clique $K$ in $G-y$ of size at least three containing $x$.
We choose $K$ such that $|K|$ is maximum. Let $t=|K|$.
So $t\geq 3$.
Since $x$ and $y$ are non-adjacent by Lemma \ref{lem 8}, $y \not \in K$.
So there exists a component $C$ of $G-K$ containing $y$.

Suppose $V(C)=\{y\}$.
Then $N_G(y)\subseteq K-x$.
Let $Y=N_G(y)\cap K$, and let $m=|Y|$.
Since $G$ is $2$-connected, we have $m\geq 2$.
For each vertex $v\in K$, let $\mathscr{D}_v$ denote the family of components $D\neq C$ of $G-K$ such that $v\in N_G(D)$.
Let $\mathscr{D}=\bigcup_{v\in K}\mathscr{D}_v$, $\mathscr{D}'=\bigcup_{v\in Y}\mathscr{D}_v$ and $\mathscr{D}''=\mathscr{D}-\mathscr{D}'$.
If there is a vertex $v\in K-x$ such that $\mathscr{D}_v=\emptyset$ or there exists some $D\in \mathscr{D}_v$ with $|V(D)|=1$,
then $t\geq k+1$, from which one can easily find $k$ paths of lengths $2,3,\ldots,k+1$ from $x$ to $y$ in $G[K\cup \{y\}]$, a contradiction.
So for every $v\in K-x$, $\mathscr{D}_v\neq\emptyset$ and $|V(D)|\geq2$ for every $D\in\mathscr{D}_v$.

Suppose that there exists some $D\in\mathscr{D}'\backslash \mathscr{D}_x$.
Let $v$ be a vertex in $N_G(D)\cap Y$ such that $D \in \mathscr{D}_v$.
Since $G$ is $2$-connected, $N_G(D)-\{v\} \neq\emptyset$.
Let $G_1$ be the graph obtained from $G[N_G[D]]$ by contracting $N_G(D)-\{v\}$ into a new vertex $u_1$.
Since $D \not \in \mathscr{D}_x$, we see $|N_G(D)-\{v\}|\leq t-2$.
So $(G_1,u_1,v)$ is $2$-connected and has minimum degree at least $k-t+4$.
By the minimality of $G$, $G_1$ contains $k-t+3$ admissible paths from $u_1$ to $v$.
Hence, $G[V(D) \cup K]$ contains $k-t+3$ admissible paths $P_i$ from a vertex $p_i\in N_G(D)-\{v\}$ to $v$ internally disjoint from $K$ for $i\in [k-t+3]$.
Since $K$ is a clique, $K-v$ contains $t-2$ paths from $x$ to $p_i$ with lengths $1,2,\ldots,t-2$, respectively.
By Lemma \ref{lemconcatenat}, by concatenating each of these paths with $P_i\cup \{vy\}$, we obtain $k$ admissible paths from $x$ to $y$ in $G$, a contradiction.

Hence $\mathscr{D}'\subseteq\mathscr{D}_x$.
Let $G_2$ be the graph obtained from $G-y$ by contracting $Y$ into a new vertex $u_2$.
Let $K'=G_2[(K-Y)\cup\{u_2\}]$. Then $K'$ is a complete graph of order $t-m+1\geq 2$ in $G_2$.
Any component $D\neq C$ of $G-K$ in $G$ is also a component of $G_2-V(K')$ in $G_2$.
If $D\in\mathscr{D}'$, then $D$ is adjacent in $G_2$ to both $x$ and $u_2$ (since $\mathscr{D}' \subseteq \mathscr{D}_x$);
otherwise $D\in\mathscr{D}''$ and $D$ is adjacent to at least two vertices of $K'-u_2$ in $G_2$ since $G$ is $2$-connected.
Since $N_G(y)$ is a clique in $G$, we have that $G-y$ is $2$-connected.
Since $\mathscr{D}'\subseteq\mathscr{D}_x$,
$(G_2,x,u_2)$ is 2-connected and has minimum degree at least $k-m+2$.
By the minimality of $G$, $G_2$ contains $k-m+1$ admissible paths from $x$ to $u_2$.
Hence, $G-y$ contains $k-m+1$ admissible paths $P_i$ from $x$ to a vertex $p_i\in Y$ for $i\in [k-m+1]$ internally disjoint from $Y$.
Since $G[Y\cup\{y\}]$ is complete, $G[Y\cup\{y\}]$ contains $m$ paths from $p_i$ to $y$ with lengths $1,2,\ldots,m$, respectively.
By Lemma~\ref{lemconcatenat}, we obtain $k$ admissible paths from $x$ to $y$ in $G$, a contradiction.

Hence $|V(C)|\geq 2$.
If $C$ is $2$-connected, then let $B=C$ and $b=y$;
otherwise let $B$ be an end-block of $C$ with cut-vertex $b$ such that $y\notin V(B)-\{b\}$.
Suppose $B$ is an edge $vb$. Then $v$ has at least $k$ neighbours in $K$.
Since $K$ is a clique, we can find $k$ consecutive paths from $x$ to $v$ in $G[K \cup\{v\}]$.
Concatenating each of these paths with a fixed path in $C$ from $v$ to $y$, we find $k$ admissible paths from $x$ to $y$, a contradiction.

Hence $B$ is $2$-connected. Let $P$ be a path in $C-V(B-b)$ from $b$ to $y$.
Since $G$ is $2$-connected, we have $N_G(B-b)\cap K \neq\emptyset$.

Suppose that $N_G(B-b)\cap(K-\{x\})\neq\emptyset$.
Let $G_3$ be the graph obtained from $G[V(B)\cup(N_G(B-b)\cap(K-\{x\}))]$ by contracting $N_G(B-b)\cap(K-\{x\})$ into a vertex $u_3$.
By the maximality of $K$, every vertex in $V(B-b)$ is adjacent to at most $t-1$ vertices in $K$.
Then $(G_3,u_3,b)$ is $2$-connected and has minimum degree at least $k-t+3$.
By the minimality of $G$, $G_3$ contains $k-t+2$ admissible paths from $u_3$ to $b$.
Hence, $G[V(B)\cup(N_G(B-b)\cap(K-\{x\}))]$ contains $k-t+2$ admissible paths $P_i$
from some vertex $p_i\in N_G(B-b)\cap(K-\{x\})$ to $b$ internally disjoint from $K$ for $i\in [k-t+2]$.
Note that for each $i$, $G[K]$ contains $t-1$ paths from $x$ to $p_i$ with lengths $1,2,\ldots,t-1$, respectively.
By Lemma~\ref{lemconcatenat}, by concatenating each of these paths with $P_i\cup P$, we obtain $k$ admissible paths from $x$ to $y$ in $G$, a contradiction.

Therefore $N_G(B-b)\cap K=\{x\}$.
Then the rooted graph $(G[V(B)\cup\{x\}],x,b)$ is $2$-connected and has minimum degree at least $k+1$.
By the minimality of $G$, $G[V(B)\cup\{x\}]$ contains $k$ admissible paths from $x$ to $b$.
By concatenating each of these paths with $P$, we obtain $k$ admissible paths from $x$ to $y$, a contradiction.

This proves that there is no clique in $G-y$ of size at least three containing $x$.
Similarly, there is no clique in $G-x$ of size at least three containing $y$, completing the proof of Lemma~\ref{case1}.
\end{proof}

In the rest of this section, by symmetry between $x$ and $y$, we may assume that $d_G(x)\leq d_G(y)$.

\begin{lem}\label{case2}
$G-y$ has a cycle of length four containing $x$.
\end{lem}

\begin{proof}
Suppose that $x$ is not contained in any cycle of length four in $G-y$.
Then
\begin{align}\label{equ:C4-free}
|N_G(v)\cap N_G(x)|\leq 1 \mbox{ for every } v\in V(G)-\{x,y\}.
\end{align}
Let $G_1$ be the graph obtained from $G$ by contracting $N_G[x]$ into a new vertex $x_1$.
By \eqref{equ:C4-free}, $G_1$ is connected and the minimum degree of $(G_1,x_1,y)$ is at least $k+1$.
If $G_1$ is not $2$-connected, then $x_1$ is the unique cut-vertex of $G_1$ and we let $B$ be the end-block of $G_1$ containing $x_1$ and $y$;
otherwise $G_1$ is $2$-connected and let $B=G_1$.

Suppose that $B$ is not an edge. Then $(B,x_1,y)$ is $2$-connected and has minimum degree at least $k+1$.
By the minimality of $G$, $B$ contains $k$ admissible paths from $x_1$ to $y$.
Then $G-x$ contains $k$ admissible paths $P_i$ from a vertex $p_i\in N_G(x)$ to $y$ for all $i\in [k]$.
By concatenating each of these paths with $xp_i$, we obtain $k$ admissible paths from $x$ to $y$ in $G$, a contradiction.

Therefore $B$ is an edge. Since $d_G(x)\leq d_G(y)$, we conclude that $N_G(x)=N_G(y)$.
By Lemma \ref{case1} and $k\geq 3$, we see $V(G)\neq N_G[x]\cup\{y\}$.
So there exists a component $D$ of $G-N_G(x)$ not containing $x$ and $y$.
Since $G$ is $2$-connected, we have $|N_G(D)|\geq 2$.
Fix a vertex $u$ in $N_G(D)$.
Let $G_2$ be the graph obtained from $G[N_G[D]]$ by contracting $N_G(D)-\{u\}$ into a new vertex $v$.
Then by \eqref{equ:C4-free}, $(G_2,u,v)$ is $2$-connected and has minimum degree at least $k+1$.
By the minimality of $G$, $G_2$ contains $k$ admissible paths from $u$ to $v$.
So $G-\{x,y\}$ contains $k$ admissible paths $P_i$ from $u$ to some vertex $p_i\in N_G(x)-\{u\}$ for $i\in [k]$.
By concatenating each of these paths with $xu$ and $p_iy$, we obtain $k$ admissible paths from $x$ to $y$ in $G$, a contradiction.
\end{proof}

\begin{lem}\label{case3}
Let $C=xx_1ax_2x$ be a cycle of length four in $G-y$.
Then every vertex in $V(G)-(V(C)\cup\{y\})$ is not adjacent in $G$ to all of $x_1,x_2,a$.
\end{lem}

\begin{proof}
Suppose to the contrary that there exists a vertex $v \in V(G)-(V(C) \cup \{y\})$ adjacent in $G$ to all of $x_1,x_2,a$.
Let $K$ be a maximal clique in $G-\{x,y,x_1,x_2\}$ such that $a\in K$ and every vertex in $K$ is adjacent to both of $x_1$ and $x_2$.
Let $t=|K|$. So $t\geq 2$.
We have the following two facts:
\begin{itemize}
	\item [(a)] for any $u\in K$, $G[V(C)\cup K]$ contains $t+1$ admissible paths from $x$ to $u$ of lengths $2,3,...,t+2$, respectively;
	\item [(b)] for any $i\in [2]$, $G[V(C)\cup K]$ contains $t$ admissible paths from $x$ to $x_i$ of lengths $3,4,...,t+2$, respectively.
\end{itemize}
Let $F$ be the component of $G-(V(C)\cup K)$ containing $y$.

Suppose $V(F)=\{y\}$.
Then $N_G(y)\subseteq V(K)\cup\{x_1,x_2\}$.
Since $G$ is $2$-connected, we have $|N_G(y)|\geq2$.
If $N_G(y)\neq\{x_1,x_2\}$, then there exists a triangle containing $y$ in $G-x$, contradicting Lemma~\ref{case1}.
Therefore $N_G(y)=\{x_1,x_2\}$.
Since $d_G(x)\leq d_G(y)$, $N_G(x)=N_G(y)=\{x_1,x_2\}$.
Let $G'=G-\{x,y\}$. It is clear that $(G',x_1,x_2)$ is $2$-connected and has minimum degree at least $k+1$.
By the minimality of $G$, $G'$ contains $k$ admissible paths from $x_1$ to $x_2$.
By concatenating each of these paths with $xx_1$ and $x_2y$, $G$ contains $k$ admissible paths from $x$ to $y$, a contradiction.

So $\lvert V(F) \rvert \geq 2$.
If $F$ is $2$-connected, let $B=F$ and $b=y$; otherwise let $B$ be an end-block of $F$ with cut-vertex $b$ such that $y\notin V(B)-b$.

Suppose that $B$ is an edge $vb$.
If $v$ is adjacent to $x$, then by Lemma~\ref{case1}, $N_G(v)\cap \{x_1,x_2\}=\emptyset$ and thus $t\geq |N_G(v)\cap K|\geq k-1$.
If $v$ is not adjacent to $x$, then by the maximality of $K$, it holds that $t+1\geq |N_G(v)\cap (K\cup \{x_1,x_2\})|\geq k\geq 3$.
So in both cases, we have $t\geq k-1$ and there exists some $u\in N_G(v)\cap K$.
By (a), there exist $k$ admissible paths from $x$ to $y$ in $G$, a contradiction.

Therefore $B$ is $2$-connected.
Let $P$ be a path in $F-V(B-b)$ from $b$ to $y$.

Suppose that $N_G(B-b)\cap K\neq\emptyset$.
Let $G_1$ be the graph obtained from $G[V(B)\cup(N_G(B-b)\cap K)]$ by contracting $N_G(B-b)\cap K$ into a new vertex $u_1$.
Let us consider the degree of any $v\in V(B-b)$ in $G_1$.
If $v$ is adjacent to both $x_1,x_2$, then by Lemma~\ref{case1} and the maximality of $K$, $v$ is not adjacent to $x$ and is adjacent to at most $t-1$ vertices in $K$,
implying that $d_{G_1}(v)\geq k+1-t$;
if $v$ is adjacent to exactly one of $x_1,x_2$, then $v$ is not adjacent to $x$ and thus $d_{G_1}(v)\geq k+1-t$;
if $v$ is adjacent to none of $x_1,x_2$, then $v$ may be adjacent to $x$ and all vertices in $K$, which also shows that $d_{G_1}(v)\geq k+1-t$.
So $(G_1,u_1,b)$ is $2$-connected and has minimum degree at least $k-t+1$.
By the minimality of $G$, $G_1$ contains $k-t$ admissible paths from $u_1$ to $b$.
Hence, $G$ contains $k-t$ admissible paths $P_i$ from a vertex $p_i\in N_G(B-b)\cap K$ to $b$ for $i\in [k-t]$ internally disjoint from $V(C)\cup K$.
By (a), $G[V(C)\cup K]$ contains $t+1$ paths from $x$ to $p_i$ with lengths $2,3,\ldots,t+2$, respectively.
By Lemma~\ref{lemconcatenat}, concatenating each of these path with $P_i\cup P$ leads to $k$ admissible paths from $x$ to $y$, a contradiction.

Therefore, $N_G(B-b)\subseteq\{x,x_1,x_2,b\}$. Since $G$ is $2$-connected, $N_G(B-b)\cap\{x,x_1,x_2\}\neq\emptyset$.
If $x_1\in N_G(B-b)$, then $(G[B\cup\{x_1\}],x_1,b)$ is $2$-connected and has minimum degree at least $k$ by Lemma \ref{case1}.
By the minimality of $G$, $G[B\cup\{x_1\}]$ contains $k-1$ admissible paths from $x_1$ to $b$.
By (b), there are $t$ admissible paths from $x$ to $x_1$ in $G[C\cup K]$.
By concatenating each of the above paths with $P$, we obtain $k-1+t-1\geq k$ admissible paths from $x$ to $y$ in $G$, a contradiction.
By symmetry between $x_1$ and $x_2$, this shows that $x_1, x_2\notin N_G(B-b)$. So $N_G(B-b)=\{x,b\}$.
Then $(G[B\cup\{x\}],x,b)$ is a $2$-connected rooted graph with minimum degree at least $k+1$,
from which one can obtain $k$ admissible paths from $x$ to $y$ by the minimality of $G$, a contradiction.
This completes the proof of Lemma \ref{case3}.
\end{proof}

\begin{lem}\label{casebi} There exists a positive integer $s$ and an induced complete bipartite subgraph $Q$ with bipartition $(Q_1,Q_2)$ in $G$ satisfying that
	\begin{enumerate}
		\item $x\in Q_2, y\notin V(Q),|Q_1|\geq|Q_2|=s+1\geq 2$, and
		\item for every $v\in V(G)-(V(Q)\cup\{y\})$, 
		\begin{enumerate}
			\item $|N_G(v)\cap Q|\leq s+1$, $|N_G(v)\cap Q_1|\leq s+1$, $|N_G(v)\cap Q_2|\leq s$, and
			\item if $v$ is adjacent to both of $Q_1$ and $Q_2$, then $|N_G(v)\cap Q_1|=1$.
		\end{enumerate}
	\end{enumerate}
\end{lem}

\begin{proof}
By Lemma \ref{case2} there exists a 4-cycle in $G-y$ containing $x$.
Thus there exists a complete bipartite subgraph $Q$ of $G-y$ with bipartition $(Q_1,Q_2)$ such that $x\in Q_2, y\notin V(Q)$ and $|Q_1|\geq|Q_2|\geq 2$.
We choose $Q$ so that $|Q_2|$ is maximum and subject to this, $|Q_1|$ is maximum.
Let $s$ be a positive integer such that $\lvert Q_2 \rvert = s+1$.

We claim that such $Q$ and $s$ satisfy the conclusion of this lemma.
Statement 2(b) holds by Lemmas \ref{case1} and \ref{case3}.
By the choice of $Q$, for every $v\in V(G)-(V(Q)\cup\{y\})$, $|N_G(v)\cap Q_1|\leq s+1$ and $|N_G(v)\cap Q_2|\leq s$.
This together with Statement 2(b), we know Statement 2(a) holds.
By Lemmas \ref{case1} and \ref{case3}, $Q$ is an induced subgraph in $G$.
The proof of Lemma~\ref{casebi} is completed.
\end{proof}

Throughout the remaining of the section, $Q$ and $s$ denote the induced complete bipartite subgraph and the positive integer $s$ promised by Lemma \ref{casebi}, and let $C$ be the component of $G-V(Q)$ containing $y$.

There are two possibilities for the size of $C$: $|V(C)|=1$ or $|V(C)|\geq 2$.
We now split the rest of the proof into two subsections based on these two cases.
We shall derive a contradiction in each subsection and hence show that $G$ is a not a counterexample to complete the proof of Theorem \ref{mainthm}.

\subsection{$|V(C)|=1$}\label{subsec:|C|=1}

In this case we have $V(C)=\{y\}$.
By Lemma \ref{lem 8}, $xy \not \in E(G)$.
So by Lemma \ref{case1}, $y$ is adjacent to exactly one of $Q_1$ and $Q_2$.
Since $d_G(y)\geq d_G(x)$, we derive that $N_G(x)=N_G(y)=Q_1$ and so $G[V(Q)\cup \{y\}]$ is complete bipartite.
If $s\geq k-1$, then $G[V(Q)\cup \{y\}]$ contains $k$ admissible paths from $x$ to $y$ of lengths $2,4,\ldots,2k$, respectively, a contradiction.

So $s\leq k-2$.
This shows that $V(G)\neq V(Q)\cup\{y\}$, for otherwise every vertex in $Q_1$ has degree $s+2\leq k$ in $G$.
Hence there exists a component in $G-(V(Q)\cup \{y\})$.

Let $D$ be an arbitrary component of $G-(V(Q) \cup \{y\})$.
If there exists a vertex $v$ of $D$ of degree at most one in $D$, then by Lemma \ref{casebi}, $s+1\geq |N_G(v)\cap V(Q)|\geq k$, a contradicting that $s\leq k-2$.
So $|V(D)|\geq 2$ and every end-block of $D$ is $2$-connected.
In addition, $N_G(x)=Q_1$, so $x\notin N_G(D)$.

We claim that $N_G(D)\cap Q_1\neq\emptyset$.
Suppose to the contrary that $N_G(D)\cap Q_1=\emptyset$.
Since $G$ is $2$-connected and $x\notin N_G(D)$, we have $|N_G(D)\cap (Q_2-\{x\})|\geq 2$.
Let $u_1$ be a vertex in $N_G(D)\cap (Q_2-\{x\})$.
Let $G_1$ be the graph obtained from $G[N_G[D]]$ by contracting $N_G(D)\cap (Q_2-\{x,u_1\})$ into a new vertex $v_1$.
Therefore $(G_1,u_1,v_1)$ is $2$-connected and has minimum degree at least $k-s+3$.
By the minimality of $G$, $G_1$ contains $k-s+2$ admissible paths from $u_1$ to $v_1$.
Hence, $G-\{x,y\}$ contains $k-s+2$ admissible paths $P_i$ from $u_1$ to some vertex $p_i\in Q_2-\{x,u_1\}$ internally disjoint from $V(Q)$ for $i\in [k-s+2]$.
Let $w$ be a vertex in $Q_1$. Since $Q$ is complete bipartite, $Q-\{u_1,w\}$ contains $s-1$ paths from $x$ to $p_i$ of lengths $2,4,\ldots,2s-2$, respectively.
By Lemma \ref{lemconcatenat}, concatenating each of these paths with $P_i$ and $u_1wy$ leads to $k$ admissible paths from $x$ to $y$, a contradiction.

We claim that $N_G(D)\cap (Q_2-\{x\})\neq\emptyset$.
Suppose to the contrary that $N_G(D)\cap (Q_2-\{x\})=\emptyset$.
Since $G$ is $2$-connected and $x \not \in N_G(D)$, $|N_G(D)\cap Q_1|\geq2$. Let $u_2$ be a vertex in $N_G(D)\cap Q_1$.
Let $G_2$ be the graph obtained from $G[N_G[D]]$ by contracting $N_G(D)\cap (Q_1-\{u_2\})$ into a new vertex $v_2$.
If $|Q_1| \geq s+2$, then let $\epsilon =0$; if $|Q_1|=s+1$, then let $\epsilon=1$.
So $(G_2,u_2,v_2)$ is $2$-connected and has minimum degree at least $k-s+1+\epsilon$.
By the minimality of $G$, $G_2$ contains $k-s+\epsilon$ admissible paths from $u_2$ to $v_2$.
Hence, $G-\{x,y\}$ contains $k-s+\epsilon$ admissible paths $P_i$ from $u_2$ to some vertex $p_i\in Q_1-{u_2}$ internally disjoint from $V(Q)$ for all $i\in [k-s+\epsilon]$.
Since $Q$ is complete bipartite, $Q-u_2$ contains $s+1-\epsilon$ paths from $x$ to $p_i$ of lengths $1,3,\ldots,2(s-\epsilon)+1$, respectively.
By Lemma \ref{lemconcatenat}, concatenating each of these paths with $P_i$ and $u_2y$ leads to $k$ admissible paths from $x$ to $y$, a contradiction.
This proves the claim.

Now we claim that there is a matching of size two in $G$ between $V(D)$ and $Q_1$.
Suppose that there is no matching of size two in $G$ between $V(D)$ and $Q_1$.
Then either $|N_G(D)\cap Q_1|=1$ or $|N_G(Q_1)\cap V(D)|=1$.
In the former case, let $u_3=w_3$ be the unique vertex in $N_G(D)\cap Q_1$;
in the latter case, let $u_3$ be the unique vertex in $N_G(Q_1)\cap V(D)$ and let $w_3$ be a vertex in $Q_1$ adjacent in $G$ to $u_3$.
Recall that $N_G(D)\cap(Q_2-\{x\})\neq\emptyset$.
Let $G_3$ be the graph obtained from $G[D\cup \{u_3\}\cup(N_G(D)\cap (Q_2-\{x\}))]$ by contracting $N_G(D)\cap (Q_2-\{x\})$ into a new vertex $v_3$.
Then $(G_3,u_3,v_3)$ is $2$-connected and has minimum degree at least $k-s+2$.
By the minimality of $G$, $G_3$ contains $k-s+1$ admissible paths from $u_3$ to $v_3$.
Hence, $G-y$ contains $k-s+1$ admissible paths $P_i$ from $u_3$ to some vertex $p_i\in Q_2-\{x\}$ internally disjoint from $V(Q)$ for $i\in [k-s+1]$.
Since $Q$ is complete bipartite, $Q-w_3$ contains $s$ paths from $x$ to $p_i$ of lengths $2,4,\ldots,2s$, respectively.
By Lemma \ref{lemconcatenat}, concatenating each of these paths with $P_i$ and $u_3w_3y$,
we obtain $k$ admissible paths from $x$ to $y$ in $G$.
This contradiction completes the proof of the claim.

Suppose that $D$ is not $2$-connected and there exists an end-block $B$ of $D$ with cut-vertex $b$ such that $N_G(B-b)\cap V(Q)\subseteq Q_2-\{x\}$.
Recall that every end-block of $D$ is $2$-connected.
So $B$ is $2$-connected.
Let $G_4$ be the graph obtained from $G[V(B)\cup (N_G(B-b)\cap (Q_2-\{x\}))]$ by contracting $N_G(B-b)\cap (Q_2-\{x\})$ into a new vertex $v_4$.
Then $(G_4,b,v_4)$ is $2$-connected and has minimum degree at least $k-s+2$.
By the minimality of $G$, $G_4$ contains $k-s+1$ admissible paths from $b$ to $v_4$.
Hence, $G$ contains $k-s+1$ admissible paths $P_i$ from $b$ to some vertex $p_i\in Q_2-\{x\}$ internally disjoint from $V(Q)$ for $i\in [k-s+1]$.
Since $N_G(D)\cap Q_1\neq \emptyset$,
there exists a path $R$ in $G[(D-V(B-b))\cup Q_1]$ from $b$ to some vertex $a\in Q_1$ internally disjoint from $V(B)\cup V(Q)$.
Since $Q$ is complete bipartite, $Q-a$ contains $s$ paths from $x$ to $p_i$ with fixed lengths $2,4,\ldots,2s$, respectively.
By Lemma~\ref{lemconcatenat}, concatenating each of these paths with $P_i\cup R\cup ay$ leads to $k$ admissible paths from $x$ to $y$ in $G$, a contradiction.

Therefore, either $D$ is $2$-connected, or every end-block $B$ of $D$ with cut-vertex $b$ satisfies that $N_G(B-b)\cap Q_1\neq\emptyset$.

We claim $|Q_1|=s+1$. Suppose to the contrary that $|Q_1|\geq s+2$.
Recall that there exists a matching $M$ of size two in $G$ between $V(D)$ and $Q_1$.
So there exists a vertex $u_5 \in N_G(D)\cap Q_1$ incident with an edge in $M$ such that $N_G(D)\cap (Q_1-\{u_5\})\neq\emptyset$.
Let $G_5$ be the graph obtained from $G[V(D)\cup (N_G(D)\cap Q_1)]$ by contracting $N_G(D)\cap (Q_1-\{u_5\})$ into a new vertex $v_5$.
Since $M$ is a matching of size two in $G$ between $V(D)$ and $Q_1$, if $D$ is $2$-connected, then $(G_5,u_5,v_5)$ is $2$-connected; if $D$ is not $2$-connected, then every end-block of $D$ has a non-cut vertex adjacent in $G_5$ to one of $u_5,v_5$, so $(G_5,u_5,v_5)$ is $2$-connected.
Moreover, by Lemma \ref{casebi}, $G_5$ has minimum degree at least $k-s+1$.
By the minimality of $G$, $G_5$ contains $k-s$ admissible paths from $u_5$ to $v_5$.
Hence, $G-y$ contains $k-s$ admissible paths $P_i$ from $u_5$ to $p_i\in V(Q_1-u_5)$ internally disjoint from $V(Q)$ for $i\in [k-s]$.
Since $|Q_1|\geq s+2$, $Q-u_5$ contains $s+1$ paths from $x$ to $p_i$ of lengths $1,3,\ldots,2s+1$, respectively.
By Lemma \ref{lemconcatenat}, concatenating each of these paths with $P_i\cup u_5y$,
we obtain $k$ admissible paths from $x$ to $y$ in $G$, a contradiction. This proves that $|Q_1|=s+1$.

Suppose that $s=1$.
Denote $Q_1$ by $\{u,v\}$.
As $N_G(x)=N_G(y)=Q_1$, it is clear that $(G-\{x,y\},u,v)$ is $2$-connected and has minimum degree at least $k+1$.
By the minimality of $G$, there are $k$ admissible paths from $u$ to $v$ in $G-\{x,y\}$,
which can be easily extended to $k$ admissible paths from $x$ to $y$ in $G$, a contradiction.

Therefore we have $s \geq 2$.
Let $w$ be a vertex in $Q_2-x$.
Since $s\leq k-2$, $w$ is adjacent in $G$ to least two vertices in $V(G)-(V(Q)\cup \{y\})$.
So there exists a non-empty set $\mathscr{D}$ of all components in $G-(V(Q)\cup \{y\})$ adjacent to $w$.
Since every member of $\mathscr{D}$ is a component of $G-(V(Q) \cup \{y\})$, for every $D' \in \mathscr{D}$, either $D'$ is $2$-connected or every end-block of $D'$ has a non-cut-vertex adjacent to $Q_1$.

Let $H=\bigcup_{D'\in\mathscr{D}}V(D')$.
Since every member $D'$ of $\mathscr{D}$ is a component of $G-(V(Q) \cup \{y\})$, there exists a matching $M_{D'}$ of size two in $G$ between $V(D')$ and $Q_1$, so we have $|N_G(H)\cap Q_1|\geq2$.
Let $u_6$ be a vertex in $N_G(H)\cap Q_1$ incident with an edge in $M_{D_0}$ for some $D_0 \in \mathscr{D}$.
Let $G_6$ be the graph obtained from $G[N_G[H]]$ by deleting $Q_2-\{x,w\}$ and contracting $Q_1-u_6$ into a new vertex $v_6$.

We claim that $(G_6,u_6,v_6)$ is $2$-connected.
Let $G'=G_6+u_6v_6$.
We shall prove that $G'$ is $2$-connected.
It suffices to show that for every $D' \in \mathscr{D}$, $G'[V(D') \cup \{u_6,v_6,w\}]$ is $2$-connected.
Suppose to the contrary that there exists $D' \in \mathscr{D}$ such that $G'[V(D') \cup \{u_6,v_6,w\}]$ is not $2$-connected.
Note that $G'[\{u_6,v_6,w\}]$ is isomorphic to $K_3$ and every end-block of $D'$ is adjacent to $\{u_6,v_6\}$.
So $G'$ is connected, and there exists a cut-vertex of $c$ of $G'[V(D') \cup \{u_6,v_6,w\}]$ such that either $c \in \{u_6,v_6,w\}$ or $c$ is a cut-vertex of $D'$.
Since $V(D')$ is adjacent to $w$ and $\{u_6,v_6\}$, if $c \in \{u_6,v_6,w\}$, then the component of $G'[V(D') \cup \{u_6,v_6,w\}]-c$ containing $V(D')$ also contains $\{u_6,v_6,w\}$, so this component contains equals $G'[V(D') \cup \{u_6,v_6,w\}]-c$, a contradiction.
So $c$ is a cut-vertex of $D'$.
But every component of $D'-c$ contains a non-cut-vertex of $D'$ in an end-block of $D'$, so it is adjacent to $\{u_6,v_6\}$, and hence there exists a component of $G'[V(D') \cup \{u_6,v_6,w\}]-c$ contains every component of $D'-c$ and $\{u_6,v_6,w\}$.
This shows that $G'[V(D') \cup \{u_6,v_6,w\}]-c$ is connected, a contradiction.
So $(G_6,u_6,v_6)$ is $2$-connected.

Now we show the minimum degree of $(G_6,u_6,v_6)$ is at least $k-s+2$.
Let $v\in V(G_6)-\{u_6,v_6,w\}$.
Then either $N_G(v)\cap Q\subseteq Q_1$, $N_G(v)\cap Q\subseteq Q_2-x$ or $v$ is adjacent to both of $Q_1$ and $Q_2-x$.
By Lemma \ref{casebi}, in either case we can derive that $d_{G_6}(v)\geq k-s+2$.
In addition, since $|Q_1|=s+1$, $d_{G_6}(w)\geq k-s+2$.
Hence, indeed the minimum degree of $(G_6,u_6,v_6)$ is at least $k-s+2$.

By the minimality of $G$, $G_6$ contains $k-s+1$ admissible paths from $u_6$ to $v_6$.
Hence, $G[N_G[H]]$ contains $k-s+1$ admissible paths $P_i$ from $u_6$ to some vertex $p_i\in Q_1-u_6$ internally disjoint from $V(Q)-w$ for $i\in [k-s+1]$.
Note that $P_i$ possibly contains $w$.
Since $Q$ is complete bipartite, $Q-\{u_6,w\}$ contains $s$ paths from $x$ to $p_i$ of lengths $1,3,\ldots,2s-1$, respectively.
By Lemma \ref{lemconcatenat}, by concatenating each of these paths with $P_i\cup u_6y$, we obtain $k$ admissible paths from $x$ to $y$ in $G$, a contradiction.
This finishes the proof of Subsection \ref{subsec:|C|=1}.

\subsection{$|V(C)|\geq 2$}
We first show that no vertex in $C-y$ has degree one in $C$.
Suppose to the contrary that there exists $v\in V(C-y)$ with degree one in $C$.
By Lemma \ref{casebi}, $s+1\geq |N_G(v)\cap V(Q)|\geq k$.
If $N_G(v)\cap Q_1\neq\emptyset$, then there are $k$ paths from $x$ to $v$ in $G[Q\cup\{v\}]$ of lengths $2,4,\ldots,2k$, respectively.
If $N_G(v)\cap Q_1=\emptyset$, then $N_G(v)\cap V(Q)\subseteq Q_2$, so $s\geq |N_G(v)\cap V(Q)|\geq k$ by Lemma \ref{casebi}, and hence there are $k$ paths from $x$ to $v$ in $G[Q\cup\{v\}]$ of lengths $3,5,\ldots,2k+1$, respectively.
In both cases, by concatenating each of these path with a path from $v$ to $y$ in $C$, we obtain $k$ admissible paths from $x$ to $y$ in $G$, a contradiction.
So no vertex in $C-y$ has degree one in $C$.
In particular, every end-block of $C$ is $2$-connected, except possibly an end-block consisting of $y$ and its unique neighbor in $C$.

We say a block of $C$ is a {\it feasible} block if it is an end-block of $C$ such that either it equals $C$, or $y$ is not a non-cut-vertex of this block.
Note that feasible blocks exist, since either $C$ has no cut-vertex, or $C$ contains at least two end-blocks.

Let $B$ be an arbitrary feasible block.
If $C$ is $2$-connected, then let $b=y$; otherwise let $b$ be the cut-vertex of $C$ contained in $B$.

We claim that $N_G(B-b)\subseteq Q_2\cup \{b\}$.
Suppose to the contrary that $N_G(B-b)\cap Q_1\neq\emptyset$.
Let $G_1$ be the graph obtained from $G[V(B)\cup (N_G(B-b)\cap Q_1)]$ by contracting $N_G(B-b)\cap Q_1$ into a new vertex $x_1$.
So $(G_1,x_1,b)$ is $2$-connected and has minimum degree at least $k-s+1$ by Lemma \ref{casebi}.
By the minimality of $G$, $G_1$ has $k-s$ admissible paths from $x_1$ to $b$.
Therefore there are $k-s$ admissible paths $P_i$ from some vertex $p_i\in N_G(B-b)\cap Q_1$ to $b$ internally disjoint from $V(Q)$ for $i\in [k-s]$.
Also $Q$ contains $s+1$ paths from $x$ to $p_i$ of fixed lengths $1,3,\ldots,2s+1$, respectively.
By Lemma \ref{lemconcatenat}, by concatenating each of these paths with $P_i$ and a fixed path in $C-V(B-b)$ from $b$ to $y$,
we obtain $k$ admissible paths from $x$ to $y$ in $G$, a contradiction.
This proves $N_G(B-b)\subseteq Q_2\cup \{b\}$.

Next we show that $s=1$ and $N_G(B-b)\cap V(Q)=Q_2$.
Let $R$ be a path in $C-V(B-b)$ from $b$ to $y$.
If $N_G(B-b)\cap Q_2=\{x\}$, then $(N_G[B],x,b)$ is $2$-connected and has minimum degree at least $k+1$, so by the minimality of $G$, $G[V(B)\cup\{x\}]$ contains $k$ admissible paths from $x$ to $b$, and hence concatenating each of them with $R$ leads to $k$ admissible paths from $x$ to $y$ in $G$, a contradiction.
So $N_G(B-b)\cap (Q_2-\{x\}) \neq \emptyset$.
Let $G_2$ be the graph obtained from $G[V(B)\cup (N_G(B-b)\cap (Q_2-\{x\}))]$ by contracting $N_G(B-b)\cap (Q_2-\{x\})$ into a new vertex $x_2$.
If $s\geq 2$ or $N_G(B-b)\cap V(Q)\subseteq Q_2- \{x\}$,
using the facts that $N_G(B-b)\subseteq Q_2\cup \{b\}$ and $|N_G(v)\cap Q_2|\leq s$ for any $v\in V(B-b)$ (the latter one is from Lemma~\ref{casebi} 2(a)),
one can verify that $(G_2,x_2,b)$ is $2$-connected and has minimum degree at least $k-s+2$.
By the minimality of $G$, $G_2$ has $k-s+1$ admissible paths from $x_2$ to $b$.
So there are $k-s+1$ paths $P_i$ from some vertex $p_i\in N_G(B-b)\cap (Q_2-\{x\})$ to $b$ internally disjoint from $V(Q)$ for $i\in [k-s+1]$.
Also $Q$ contains $s$ admissible paths from $x$ to $p_i$ of lengths $2,4,\ldots,2s$, respectively.
By Lemma \ref{lemconcatenat}, by concatenating each of these paths with $P_i$ and $R$, we obtain $k$ admissible paths from $x$ to $y$, a contradiction.
This shows that $s =1$ and $N_G(B-b)\cap V(Q)=Q_2$.

We denote $Q_2$ by $\{x,a\}$.
So $N_G(B-b)\cap V(Q)=Q_2=\{x,a\}$.

\medskip

{\bf Case 1.} $N_G(C-y)\cap Q_1=\emptyset$.

\medskip

Since $N_G(B-b)\cap V(Q)=Q_2=\{a,x\}$, we have that $(G[V(B)\cup\{a\}],a,b)$ is $2$-connected and has minimum degree at least $k$.
By the minimality of $G$, $G[V(B)\cup\{a\}]$ contains $k-1$ admissible paths $P_1,\ldots,P_{k-1}$ from $a$ to $b$.
Let $Y$ be a path from $b$ to $y$ in $C-V(B-b)$.

For any $v\in Q_1$, if $N_G(v)\subseteq Q_2\cup\{y\}$, then the degree of $v$ in $G$ is three, so $k \leq 2$, contradicting Lemma \ref{lem 8}. 
Therefore, there exists a component $D$ of $G-V(Q\cup C)$ adjacent to $v$.
Since $N_G(C-y)\cap Q_1=\emptyset$, $N_G(Q_1) \cap V(C) \subseteq \{y\}$.
So $(G-V(C),x,a)$ is $2$-connected and has minimum degree at least $k$.
By the minimality of $G$, there are $k-1$ admissible paths $R_1,..., R_{k-1}$ from $x$ to $a$ in $G-V(C)$.
Then by Lemma \ref{lemconcatenat}, $R_i\cup P_j\cup Y$ for all $i,j\in [k-1]$ give at least $2k-3\geq k$ admissible paths from $x$ to $y$, a contradiction.
This completes the proof of Case 1.

\medskip

{\bf Case 2.} $N_G(C-y)\cap Q_1\neq\emptyset$.

\medskip

If $C$ is $2$-connected, then $C=B$ and $y=b$, contradicting $N_G(B-b)\cap V(Q)=\{x,a\}$.
So $C$ is not $2$-connected.
Let $B_1,B_2,\ldots, B_t$ be all end-blocks of $C$ with cut-vertices $b_1,b_2,\ldots,b_t$, respectively.
Note that $t\geq 2$.

Suppose that $y\notin\bigcup_{i=1}^t(V(B_i)-\{b_i\})$.
So for every $i \in [t]$, $B_i$ is a feasible block, and hence $N_G(B_i-b_i) \cap V(Q)=\{x,a\}$ which is disjoint from $Q_1$.
Since $N_G(C-y)\cap Q_1\neq\emptyset$,
there is a vertex $w$ in $V(C)-(\bigcup_{i=1}^t(V(B_i)-\{b_i\}) \cup \{y\})$ such that $N_G(w)\cap Q_1\neq\emptyset$.
Let $c$ be a vertex in $N_G(w)\cap Q_1$.
Using the block structure of $C$, there exist two end-blocks $B_m, B_n$ for $1\leq m<n\leq t$,
such that there are two disjoint paths $L_1, L_2$ from $b_m$ to $w$ and from $b_n$ to $y$ internally disjoint from $V(B_n)\cup V(B_m)$, respectively.
Since $B_m$ and $B_n$ are feasible, $N_G(B_m-b_m) \cap V(Q) = \{x,a\} = N_G(B_n-b_n) \cap V(Q)$.
So both of $(G[V(B_m)\cup\{x\}],x,b_m)$ and $(G[V(B_n)\cup\{a\}],a,b_n)$ are $2$-connected and have minimum degree at least $k$.
By the minimality of $G$, there are $k-1$ admissible paths $P_1,\ldots,P_{k-1}$ from $x$ to $b_m$ in $G[V(B_m)\cup\{x\}]$;
and there are $k-1$ admissible paths $R_1,\ldots,R_{k-1}$ from $a$ to $b_n$ in $G[V(B_n)\cup\{a\}]$.
By Lemma \ref{lem 8}, $k\geq 3$.
So the set $\{P_i\cup L_1\cup wca\cup R_j\cup L_2: i,j\in [k-1]\}$ contains at least $2k-3\geq k$ admissible paths from $x$ to $y$ in $G$, a contradiction.

So there exists an end-block, say $B_t$, of $C$ such that $y\in V(B_t)-\{b_t\}$.
We say that a block $H$ of $C$ other than $B_1$ is a {\it hub} if $H$ is $2$-connected and contains at most two cut-vertices of $C$, and every path in $C$ from $B_1$ to $B_t$ contains all cut-vertices of $C$ contained in $V(H)$.

Suppose there exists a hub $B^*$ of $C$. 
So there exists a cut-vertex $x^*$ of $C$ contained in $B^*$ such that every path in $C$ from $b_1$ to $V(B^*)$ contains $x^*$.
If $B^*=B_t$, then let $y^*=y$; otherwise, let $y^*$ be the cut-vertex of $C$ contained in $B^*$ such that every path in $C$ from $b_t$ to $V(B^*)$ contains $y^*$.
Let $Z_0$ be a path in $C-(V(B_1-b_1)\cup V(B^*-x^*))$ from $b_1$ to $x^*$, and let $Z_1$ be a path in $C$ from $y^*$ to $y$.
Since $(G[B_1\cup \{x\}],x,b_1)$ is $2$-connected with minimum degree at least $k$, by the minimality of $G$, $G[B_1\cup \{x\}]$ contains $k-1$ admissible paths $P_1,\ldots,P_{k-1}$ from $x$ to $b_1$.
If every vertex in $V(B^*)-\{x^*,y^*\}$ has at most one neighbor in $Q$, then $(B^*,x^*,y^*)$ is $2$-connected with minimum degree at least $k$.
By the minimality of $G$, $B^*$ contains $k-1$ admissible paths $R_1,\ldots,R_{k-1}$ from $x^*$ to $y^*$.
By Lemmas \ref{lemconcatenat} and \ref{lem 8}, the set $\{P_i\cup Z_0\cup R_j \cup Z_1: i,j\in [k-1]\}$ contains least $2k-3\geq k$ admissible paths from $x$ to $y$ in $G$, a contradiction.
Therefore some vertex $w\in V(B^*)-\{x^*,y^*\}$ satisfies $|N_G(w)\cap V(Q)|\geq 2$.
Since $s=1$, we have $|N_G(w)\cap V(Q)|=2$ by Lemma \ref{casebi}.
Let $u,v$ be the vertices in $N_G(w)\cap V(Q)$.
By Lemma \ref{casebi}, either $\{u,v\} \subseteq Q_1$, or by symmetry say $u\in Q_1$ and $v\in Q_2$.
In the former case, there are two admissible paths $L_1=xua$ and $L_2=xuwva$ from $x$ to $a$;
in the latter case, since there is no triangle containing $x$ in $G-y$ by Lemma \ref{case1}, we must have $v=a$,
which also gives two admissible paths $L_1=xua$ and $L_2=xuwa$ from $x$ to $a$.
Since $(G[B_1\cup \{a\}],a,b_1)$ is $2$-connected with minimum degree at least $k$, by the minimality of $G$, there exist $k-1$ admissible paths $N_1,\ldots,N_{k-1}$ from $a$ to $b_1$ in $G[B_1\cup \{a\}]$.
Since $B^*$ is $2$-connected, there exists a path $L'$ from $x^*$ to $y^*$ in $B-w$.
By Lemma \ref{lemconcatenat}, the set $\{L_i\cup N_j\cup Z_0\cup L' \cup Z_1: i\in [2], j\in [k-1]\}$ contains $k$ admissible paths from $x$ to $y$ in $G$, a contradiction.

So there exists no hub.
In particular, $B_t$ is not $2$-connected, for otherwise $B_t$ is a hub.
Therefore $B_t=yb_t$ is an edge.
So $B_1,...,B_{t-1}$ are the all feasible blocks in $C$.
Recall that $N_G(B_i-b_i)\cap V(Q)=\{a,x\}$ for all $i\in [t-1]$, which implies $d_G(x)\geq |Q_1|+t-1$.
Since there is no triangle containing $y$ in $G-x$ by Lemma \ref{case1}, we have $d_G(y)\leq |Q_1|+1$.
Hence $|Q_1|+t-1\leq d_G(x)\leq d_G(y)\leq |Q_1|+1$.
That is, $t \leq 2$.
As $t\geq 2$, this forces $t=2$, $d_G(x)=d_G(y)=|Q_1|+1$.
In other words, there is exactly one end-block $B_1$ of $C$ other than $B_2=yb_2$, $N_G(y)=Q_1\cup \{b_2\}$ and $N_G(x)\subseteq Q_1\cup V(B_1-b_1)$.
Note that the block structure of $C$ is a path.
Since there exists no hub, every block of $C$ other than $B_1$ is an edge.
If $V(C)=V(B_1\cup B_2)$, then since $N_G(C-y)\cap Q_1\neq\emptyset$ and $N_G(B_1-b_1)\cap Q_1=\emptyset$, $b_2$ must have a neighbor in $Q_1$.
If $V(C)\neq V(B_1\cup B_2)$, then $|N_G(b_2)\cap V(C)|=2$, and since $d_G(b_2)\geq k+1\geq 4$, we have $|N_G(b_2)\cap V(Q)|\geq 2$.
Recall that $N_G(x)\subseteq Q_1\cup V(B_1-b_1)$, so $xb_2\notin E(G)$.
So in either case, $b_2$ must have a neighbor $w^*$ in $Q_1$.
But $G[\{y,b_2,w^*\}]$ is a triangle, contradicting Lemma \ref{case1}.

This completes the proof of Theorem \ref{mainthm} (and of Theorem \ref{thm 1}).
\qed

\section{Consecutive cycles} \label{sec:consecutive_cycles}
In this section, we prove Theorems \ref{nonbi 3} and \ref{chroconcy}.
This will be achieved in a unified approach, namely,
by finding optimal number of cycles of consecutive lengths in $2$-connected non-bipartite graphs (see Theorem \ref{nonseparating cycle->consecutive cycles}).

We begin by introducing a concept on cycles, which is crucial in our approach.
We say that a cycle $C$ in a connected graph $G$ is {\it non-separating} if $G-V(C)$ is connected.
The study of non-separating cycles appears in the work of Tutte \cite{T63} and is furthered explored by Thomassen and Toft \cite{TT81}.
The proof of the following lemma can be found in \cite{BV98} (though it was not formally stated).

\begin{lem}[Bondy and Vince \cite{BV98}]\label{lem:3-con->non-separating}
Every non-bipartite $3$-connected graph contains a non-separating induced odd cycle.
\end{lem}

We also need the following lemma on non-separating odd cycles from \cite{LM},
which is a slight modification of a result of Fan \cite{Fan02}.

\begin{lem}[Liu and Ma \cite{LM}]\label{lem:better non-separating cycle}
Let $G$ be a graph with minimum degree at least four.
If $G$ contains a non-separating induced odd cycle, then $G$ contains a non-separating induced odd cycle $C$, denoted by $v_0v_1...v_{2s}v_0$, such that either
	\begin{enumerate}
		\item $C$ is a triangle, or
		\item for every non-cut-vertex $v$ of $G-V(C)$, $\lvert N_G(v)\cap V(C) \rvert \le 2$, and the equality holds if and only if $N_G(v)\cap V(C)=\{v_i,v_{i+2}\}$ for some $i$, where the indices are taken under the additive group $\mathbb{Z}_{2s+1}$.
	\end{enumerate}
\end{lem}

The next lemma can be viewed as a corollary of Theorem \ref{mainthm},
which will be used for finding paths in a $2$-connected graph with three special vertices.

\begin{lem}\label{three vertices}
Let $k\geq 2$ be a positive integer. Let $G$ be a $2$-connected graph and $x,y,z$ be three distinct vertices in $G$.
If every vertex of $G$ other than $z$ has degree at least $k+1$, then $G$ contains $k-1$ admissible paths from $x$ to $y$.
\end{lem}

\begin{proof}
Since every two vertices are contained in a cycle in a $2$-connected graph, there is nothing to prove when $k=2$.
So we may assume that $k\geq 3$. Note that $G-z$ is connected and has minimum degree at least $k$.
If $G-z$ is $2$-connected, then this lemma follows from Theorem~\ref{mainthm}.
Hence we may assume that $G-z$ is not $2$-connected.

Let $B$ be an end-block of $G-z$ with cut-vertex $b$.
Since every vertex in $V(B-b)$ has degree at least $k\geq3$ in $G$, we see that $B$ is $2$-connected.
Suppose that $|V(B-b)\cap \{x,y\}|=1$. Without loss of generality, we may assume that $x\in V(B-b)$.
By Theorem~\ref{mainthm}, $B$ has $k-1$ admissible paths from $x$ to $b$.
Concatenating each of these paths with a path in $(G-z)-V(B-b)$ from $b$ to $y$ gives $k-1$ admissible paths in $G$ from $x$ to $y$.
Therefore, there exists an end-block $B'$ with cut-vertex $b'$ of $G-z$ such that $V(B'-b')\cap \{x,y\}=\emptyset$.
It follows that $N_G(B'-b')=\{b',z\}$.
Since $G$ is $2$-connected, $G$ has two disjoint paths $P_1,P_2$ internally disjoint from $V(B')$ from $x$ to $b'$ and from $y$ to $z$, respectively.
Let $u$ be a vertex in $B'-b'$ adjacent to $z$ in $G$.
By Theorem \ref{mainthm}, $B'$ has $k-1$ admissible paths $R_1,R_2,...,R_{k-1}$ from $b'$ to $u$.
Then the set $\{P_1\cup R_i\cup uz\cup P_2: i\in [k-1]\}$ contains $k-1$ admissible paths in $G$ from $x$ to $y$.
This completes the proof.
\end{proof}

We are ready to prove the main result of this section.

\begin{theo} \label{nonseparating cycle->consecutive cycles}
Let $k$ be a positive integer and $G$ be a $2$-connected graph containing a non-separating induced odd cycle.
If the minimum degree of $G$ is at least $k+1$, then $G$ contains $k$ cycles of consecutive lengths.
\end{theo}

\begin{proof}
The theorem is obvious when $k=1$.
For the case $k=2$, let $C_0$ be an induced non-separating odd cycle in $G$ and $x,y\in V(C_0)$
such that $x,y$ divide $C_0$ into two subpaths say $P_1, P_2$ of lengths differing by one.
Since $G$ has minimum degree at least three,
each of $x,y$ has at least one neighbor in $G-V(C_0)$ and thus there exists a path $L$ from $x$ to $y$ in $G[(V(G)-V(C_0)) \cup \{x,y\}]$.
Then $L\cup P_1$ and $L\cup P_2$ are two cycles of consecutive lengths in $G$.

So we may assume that $k\geq 3$. By Lemma~\ref{lem:better non-separating cycle}, there exists a non-separating induced odd cycle $C=v_0v_1...v_{2s}v_0$ in $G$ satisfying the conclusions of Lemma~\ref{lem:better non-separating cycle}.
In particular, the minimum degree of $G-V(C)$ is at least $k-1$.
Throughout the rest of the proof of this theorem, the subscripts will be taken under the additive group $\mathbb{Z}_{2s+1}$.

Suppose that $C$ is a triangle $v_0v_1v_2v_0$.
Consider the graph $G'$ obtained from $G$ by contracting $v_1$ and $v_2$ into a vertex $u$.
Then $G'$ is $2$-connected with minimum degree at least $k$.
By Theorem \ref{mainthm}, there are $k-1$ admissible paths in $G'$ from $u$ to $v_0$.
By the definition of admissible paths, each of these paths has length at least two,
so it does not contain the edge $uv_0$, and each of those paths corresponds to a path in $G-V(C)$ from $v_0$ to some $v_i\in \{v_1,v_2\}$.
Concatenating with $v_0v_i$ and $v_0v_{3-i}v_i$, these paths lead to cycles of at least $k$ consecutive lengths in $G$.

Therefore we may assume that $C$ is not a triangle and hence $s\geq2$.
For any two vertices $v_i,v_j$ in $C$, denote $C'_{i,j}$ and $C''_{i,j}$ to be the shorter and longer paths in $C$ from $v_i$ to $v_j$, respectively.

Suppose that $G-V(C)$ is $2$-connected.
We first assume that for every $v\in V(G-C)$, $|N_G(v)\cap V(C)|\leq 1$.
Then the minimum degree of $G-V(C)\geq k$.
Since $G$ has minimum degree at least $k+1\geq4$, there exist distinct vertices $x,y\in V(G-C)$ such that $xv_0,yv_s\in E(G)$.
By Theorem~\ref{mainthm}, $G-V(C)$ contains $k-1$ admissible paths $P_1,...,P_{k-1}$ from $x$ to $y$.
Note that $C'_{0,s}$ and $C''_{0,s}$ are two paths from $v_0$ to $v_s$ of lengths $s$ and $s+1$, respectively.
Concatenating each of $C'_{0,s}$ and $C''_{0,s}$ with $v_0x\cup P_i\cup yv_s$ for all $i\in [k-1]$ leads to $k$ cycles of consecutive lengths in $G$.
Hence we may assume that there exists some $u\in V(G-C)$ adjacent in $G$ to two vertices of $C$.
By Lemma~\ref{lem:better non-separating cycle}, without loss of generality, let $N_G(u)\cap V(C)=\{v_1,v_{2s}\}$.
Since $C$ is an induced cycle and $d(v_s)\geq\delta(G)\geq k+1\geq 4$, there exists a vertex $w\in V(G-C)-\{u\}$ such that $wv_s\in E(G)$.
Since $G-V(C)$ has minimum degree at least $k-1$, by Theorem~\ref{mainthm}, $G-V(C)$ contains $k-2$ admissible paths $R_1,...,R_{k-2}$ from $u$ to $w$.
Observe that $uv_1\cup C'_{1,s}, uv_{2s}\cup C'_{s,2s}, uv_{2s}\cup C''_{s,2s}$ and $uv_1\cup C''_{1,s}$
are four paths from $u$ to $v_{s}$ of lengths $s,s+1,s+2$ and $s+3$, respectively.
By concatenating each of these paths with $v_{s}w\cup R_i$ for $i\in [k-2]$,
we obtain cycles of $k+1$ consecutive lengths in $G$.

Therefore $G-V(C)$ is not $2$-connected.
Let $B$ be an end-block of $G-V(C)$ with cut-vertex $b$.
Since every vertex in $B-b$ has degree at least $k-1\geq2$ in $B$, $B$ is $2$-connected.

Suppose that $|N_G(v)\cap V(C)|\leq 1$ for every vertex $v\in V(B-b)$.
Then every vertex in $B$ other than $b$ has degree at least $k$ in $B$.
We first assume that there exist $x\in V(B-b)$ and $y\in V(G-C)-V(B-b)$ such that $v_jx, v_{j+s}y\in E(G)$ for some $j$,
then by Theorem~\ref{mainthm}, $B$ contains $k-1$ admissible paths $P_1,...,P_{k-1}$ from $x$ to $b$.
Let $P$ be a path in $G-(V(C) \cup V(B-b))$ from $b$ to $y$.
Note that $C'_{j,j+s},C''_{j,j+s}$ are two paths of lengths $s,s+1$, respectively.
Then, by concatenating each of these paths with $P_i$ and $P$, we find $k$ cycles in $G$ with consecutive lengths.
Hence, we may assume that for every integer $j$ with $0 \leq j \leq 2s$, if $v_j$ is adjacent to $V(B-b)$, then $N_G(v_{j+s}) \cap V(G-C) \subseteq V(B-b)$.
Since $G$ is $2$-connected, there is some vertex $v_{i^*}$ of $C$ adjacent in $G$ to $V(B-b)$.
Since $k+1 \geq 4$, every vertex in $V(C)$ is adjacent in $G$ to some vertex in $V(G-C)$.
So $\emptyset \neq N_G(v_{i^*+s})-V(C) \subseteq V(B-b)$.
Hence we can inductively show that $N_G(v_{i^*+rs})-V(C) \subseteq V(B-b)$ for every positive integer $r$.
Since $s$ is a generator of ${\mathbb Z}_{2s+1}$, $N_G(C)\subseteq V(B-b)$.
This implies that $b$ is a cut-vertex of $G$, contradicting the $2$-connectivity of $G$.

Therefore there exists a vertex $x\in V(B-b)$ with at least two neighbors in $V(C)$.
By Lemma \ref{lem:better non-separating cycle}, without loss of generality,
we may assume that $N_G(x)\cap V(C)=\{v_1,v_{2s}\}$.
Assume there exists some $y\in V(G-C)-V(B-b)$ such that $v_{s}y\in E(G)$.
Since every vertex in $B-b$ has degree at least $k-1$ in $B$,
by Theorem~\ref{mainthm}, $B$ contains $k-2$ admissible paths $Q_1,...,Q_{k-2}$ from $x$ to $b$.
Let $Q$ be a fixed path in $G-(V(C)\cup V(B-b))$ from $b$ to $y$.
Note that $xv_1\cup C'_{1,s}, xv_{2s}\cup C'_{s,2s}, xv_{2s}\cup C''_{s,2s}$ and $xv_1\cup C''_{1,s}$
are four paths from $x$ to $v_s$ of lengths $s,s+1,s+2$ and $s+3$, respectively.
By concatenating each of these paths with $Q_i\cup Q\cup yv_s$, we find $k+1$ cycles of consecutive lengths in $G$.
Hence we have $N_G(v_{s})\cap V(G-C)\subseteq V(B-b)$.
Since $|N_G(v_{s})\cap V(G-C)|\ge k-1\ge 2$, there exists $z\in N_G(v_s)\cap V(B)-\{x,b\}$. 
If $k\leq 4$, then using the above four paths from $x$ to $v_s$, together with $v_sz$ and a path in $B$ from $z$ to $x$, we obtain cycles of four consecutive lengths in $G$.
So we may assume $k\geq 5$. 
Note that every vertex of $B$ other than $b$ has degree at least $k-1\geq 4$ in $B$.
By Lemma \ref{three vertices}, $B$ has $k-3$ admissible paths $R_1,...,R_{k-3}$ from $x$ to $z$.
Again, concatenating each of these paths with $zv_s$ and the four paths from $x$ to $v_s$,
one can find cycles of $k$ consecutive lengths in $G$.
This completes the proof of Theorem~\ref{nonseparating cycle->consecutive cycles}.
\end{proof}

Using Theorem \ref{nonseparating cycle->consecutive cycles}, we can derive Theorems \ref{nonbi 3} and \ref{chroconcy} easily.

\medskip

{\noindent \bf Theorem \ref{nonbi 3}.}
{\it Every non-bipartite $3$-connected graph with minimum degree at least $k+1$ contains $k$ cycles of consecutive lengths.}

\begin{proof}
This theorem immediately follows from Lemma \ref{lem:3-con->non-separating} and Theorem \ref{nonseparating cycle->consecutive cycles}.
\end{proof}

We say that a graph $G$ is {\em $k$-critical}, if it has chromatic number $k$ but every proper subgraph of $G$ has chromatic number less than $k$.

We now prove Theorem \ref{chroconcy}, which we restate as the following.

\medskip

\begin{theo}
For every positive integer $k$, every graph with chromatic number at least $k+2$ contains $k$ cycles of consecutive lengths.
\end{theo}

\begin{proof}
Let $G$ be any graph with chromatic number at least $k+2$.
We may assume that $k\geq 2$, for otherwise the theorem is obvious.
Then there exists a $(k+2)$-critical subgraph $G'$ of $G$.
It is easy to see that $G'$ is $2$-connected and has minimum degree at least $k+1$.
It is known that for any integer $t\geq 4$, every $t$-critical graph contains a non-separating induced odd cycle
(the case $t=4$ was explicitly stated and proved by Krusenstjerna-Hafstr{\o}m and Toft \cite[Theorem 4]{KHT},
but their proof works for every $t\geq 4$ as well).
Therefore $G'$ contains a non-separating induced odd cycle.
By Theorem~\ref{nonseparating cycle->consecutive cycles}, we see that $G'$ (and thus $G$) contains $k$ cycles of consecutive lengths.
\end{proof}

\section{Dean's conjecture} \label{sec:Dean_proof}
In this section we prove Conjecture \ref{deanconj}, which will be divided into several lemmas.
For a brief overview of the coming proof, we would suggest readers to have a sketch on the proof of Theorem \ref{Thm:Dean2},
which is a restatement of Theorem \ref{deantheo}.

Define $K_4^-$ to be the graph obtained from $K_4$ by deleting one edge.
A {\it chord} of a cycle $C$ in a graph $G$ is an edge $e \in E(G)-E(C)$ such that the both ends of $e$ belong to $V(C)$.
For a positive integer $t \geq 4$, we define $K_{t}'$ to be the graph obtained from $K_t$ by deleting $v_1v_4$ and $v_iv_j$ for every $i \in \{1,2\}$ and $j \in \{5,6,...,t\}$, where $V(K_t) = \{v_k: 1 \leq k \leq t\}$.
Note that $K_{t}'$ contains a Hamilton cycle.
Also, for every positive integer $d$, let $K_{d,d}^-$ be the graph obtained from $K_{d,d}$ by deleting an edge.

\begin{lem} \label{usingK_4-}
Let $d$ and $t$ be integers with $d+1 \geq t \geq 5$.
Let $G$ be a graph containing a $K_4^-$ subgraph but not containing a $K_t'$ subgraph.
If $G$ is $(t-1)$-connected, then $G$ contains $d$ cycles of consecutive lengths.
\end{lem}

\begin{proof}
Let $\{x,y,a,b\}$ be a set of four vertices of $G$ inducing a $K_4^-$ subgraph, where $x$ is of degree two in this $K_4^-$ subgraph and $y$ is a neighbor of $x$.
So there exists a clique in $G-\{x,y\}$ containing $a,b$.
Let $K$ be a maximal clique in $G-\{x,y\}$ containing $a,b$.
Note that $\lvert K \rvert \leq t-3$, for otherwise $G[\{x,y\} \cup K]$ contains a $K_t'$ subgraph.
Hence $G-K$ is $2$-connected since $G$ is $(t-1)$-connected.

By the maximality of $K$, every vertex in $G-(K \cup \{x,y\})$ is adjacent in $G$ to at most $\lvert K \rvert-1$ vertices in $K$.
So $((G-K)-xy,x,y)$ is a $2$-connected rooted graph of minimum degree at least $d-\lvert K \rvert+1$.
By Theorem \ref{mainthm}, there exist $d-\lvert K \rvert$ admissible $x$-$y$ paths $P_1,P_2,...,P_{d-\lvert K \rvert}$ in $(G-K)-xy$.
Note that there exist $\lvert K \rvert+1$ $x$-$y$ paths $Q_1,Q_2,...,Q_{\lvert K \rvert+1}$ in $G[K \cup \{x,y\}]$ with consecutive lengths.
For every integers $i,j$ with $1 \leq i \leq d-\lvert K \rvert$ and $1 \leq j \leq \lvert K \rvert+1$, let $C_{i,j}$ be the cycle obtained by concatenating $P_i$ and $Q_j$.
Let ${\mathcal C} = \{C_{i,j}: 1 \leq i \leq d-\lvert K \rvert, 1 \leq j \leq \lvert K \rvert+1\}$.
If $P_1,P_2,...,P_{d-\lvert K \rvert}$ have consecutive lengths, then ${\mathcal C}$ contains $d$ cycles of consecutive lengths.
If the lengths of $P_1,P_2,...,P_{d-\lvert K \rvert}$ form an arithmatic progression of length two, then ${\mathcal C}$ contains $2d-\lvert K \rvert-2 \geq d$ cycles of consecutive lengths.
\end{proof}

\begin{lem} \label{K_3noK_4-}
Let $d \geq 3$ be an integer.
Let $G$ be a $3$-connected graph of minimum degree at least $d$.
If $G$ contains a $K_3$ subgraph but does not contain a $K_4^-$ subgraph, then $G$ contains $d$ cycles of consecutive lengths.
\end{lem}

\begin{proof}
Let $\{a,b,c\}$ be a set of three vertices of $G$ that induces a $K_3$ subgraph.
Let $G'$ be the graph obtained from $G$ by contracting the edge $bc$ into a new vertex $a^*$ and deleting resulting loops and parallel edges.
Since $G$ is $3$-connected, $G'$ is $2$-connected, so $(G'-aa^*,a,a^*)$ is a $2$-connected rooted graph.
Since $G$ does not contain a $K_4^-$ subgraph, $(G'-aa^*,a,a^*)$ has minimum degree at least $d$.
By Theorem \ref{mainthm}, there exist $d-1$ admissible $a$-$a^*$ paths $P_1,P_2,...,P_{d-1}$ in $G'-aa^*$.
So there exist paths $P_1',P_2',...,P_{d-1}'$ in $G$ such that their lengths form an arithmetic progression of common difference one or two, and for every $i$ with $1 \leq i \leq d-1$, $P_i'$ is either an $a$-$b$ path disjoint from $c$ or an $a$-$c$ path disjoint from $b$.
For every integer $i$ with $1 \leq i \leq d-1$, let $C_{i,1}=P_i'+ab$, $C_{i,2}=P_i'+ac$, $C_{i,3}=P_i' \cup abc$, and let $C_{i,4}=P_i' \cup acb$.
Then the set $\{C_{i,j}: 1 \leq i \leq d-1, 1 \leq j \leq 4\}$ contains $d$ cycles of consecutive lengths.\end{proof}

\begin{lem} \label{plus3}
Let $\ell$ be a positive integer.
Let $A$ be a subset of integers such that $\ell$ elements of $A$ form an arithmetic progression of common difference $r$, where $r \in \{1,2\}$.
	\begin{enumerate}
		\item If $r=1$ and $\ell \geq 3$, then for every integer $x$, the set $\{a+x, a+x+3: a \in A\}$ contains $\ell+3$ elements that form an arithmetic progression of common difference one.
		\item If $r=2$ and $\ell \geq 2$, then for every integer $x$, the set $\{a+x, a+x+3: a \in A\}$ contains $2\ell-2$ elements that form an arithmetic progression of common difference one.
	\end{enumerate}
\end{lem}

\begin{proof}
Let $a_1,a_2,...,a_\ell$ be $\ell$ elements of $A$ forming an arithmetic progression of common difference $r$, where $r \in \{1,2\}$.
We may assume that for every $i$ with $1 \leq i \leq \ell$, $a_i = a_1 + (i-1)r$.
Let $x$ be any integer, and let $S = \{a_i+x, a_i+x+3: 1 \leq i \leq \ell\}$.

If $r=1$ and $\ell \geq 3$, then $S = \{i: a_1+x \leq i \leq a_1+\ell-1+x+3\}$.
If $r=2$ and $\ell \geq 2$, then $S$ contains $\{i: a_1+2+x \leq i \leq a_1+2(\ell-1)+x+1\}$.
This proves the lemma.
\end{proof}

\begin{lem} \label{no_K3_nb1}
Let $d$ be an integer with $d\geq 6$.
If $G$ is a $3$-connected non-bipartite graph with minimum degree at least $d$ that does not contain a $K_3$ subgraph, then either
	\begin{enumerate}
		\item $G$ contains $d$ cycles of consecutive lengths, or
		\item $d \in \{6,7\}$ and there exists an induced cycle $C$ in $G$, denoted by $v_0v_1v_2...v_{2s}v_0$, of length at least five such that
		\begin{enumerate}
			\item $G-V(C)$ is connected but not $2$-connected,
			\item every end-block of $G-V(C)$ is $2$-connected, and
			\item for every non-cut vertex $v$ of $G-V(C)$, $\lvert N_G(v) \cap V(C) \rvert \leq 2$, and if $\lvert N_G(v) \cap V(C) \rvert = 2$, then $N_G(v) \cap V(C) =\{v_i,v_{i+2}\}$ for some $i \in {\mathbb Z}_{2s+1}$ and the indices are computed in ${\mathbb Z}_{2s+1}$. 
		\end{enumerate}
	\end{enumerate}
\end{lem}

\begin{proof}
	We may assume that $G$ does not contain $d$ cycles of consecutive lengths, for otherwise we are done.
	
	By Lemmas \ref{lem:3-con->non-separating} and \ref{lem:better non-separating cycle}, since $G$ does not contain a $K_3$ subgraph, there exists an induced odd cycle $C$ of length at least five, denoted by $v_0v_1...v_{2s}v_0$, such that for every non-cut vertex $v$ of $G-V(C)$, $\lvert N_G(v) \cap V(C) \rvert \leq 2$, and if $\lvert N_G(v) \cap V(C) \rvert = 2$, then $N_G(v) \cap V(C) =\{v_i,v_{i+2}\}$ for some $i \in {\mathbb Z}_{2s+1}$ and the indices are computed in ${\mathbb Z}_{2s+1}$.
	
	In particular, no vertex of $G-V(C)$ is of degree at most one in $G-V(C)$, since $d \geq 4$.
	So every end-block of $G-V(C)$ is $2$-connected.
	Note that for every end-block $B$ of $G-V(C)$, there exists at most one cut-vertex of $G-V(C)$ contained in $V(B)$, and if such vertex exists, we denote it by $b_B$.
	
So to prove this lemma, it suffices to prove that $d \in \{6,7\}$ and $G-V(C)$ is not $2$-connected.

We first suppose to the contrary that $G-V(C)$ is $2$-connected.

Suppose that there exists a vertex $x \in V(G)-V(C)$ adjacent in $G$ to at least two vertices in $V(C)$.
By symmetry, we may assume that $x$ is adjacent to $v_0$ and $v_2$.
Since $G$ is $3$-connected and $C$ is an induced cycle in $G$, $v_{s+1}$ is adjacent in $G$ to a vertex $y$ in $V(G)-V(C)$.
Since $\lvert V(C) \rvert \geq 5$, $v_{s+1} \not \in \{v_0,v_2\}$, so $x \neq y$.
Since $G-V(C)$ is $2$-connected, every vertex in $G-V(C)$ has degree at least $d-2$, so by Theorem \ref{mainthm}, there exist $d-3$ admissible paths $P_1,P_2,...,P_{d-3}$ in $G-V(C)$ from $x$ to $y$.
Let $Q_1,Q_2$ be the subpaths of $C$ with ends $v_{0}$ and $v_{s+1}$ of length $s+1$ and $s$, respectively.
Let $Q_3,Q_4$ be the subpaths of $C$ with ends $v_{2}$ and $v_{s+1}$ of length $s-1$ and $s+2$, respectively.
Let ${\mathcal C}=\{(P_i \cup Q_j)+xv_{0}+yv_{s+1}, (P_i \cup Q_k)+xv_{2}+yv_{s+1}: 1 \leq i \leq d-3, 1 \leq j \leq 2, 3 \leq k \leq 4\}$.
Then ${\mathcal C}$ contains $(d-3)+4-1=d$ cycles of consecutive lengths, a contradiction.

So every vertex $x \in V(G)-V(C)$ is adjacent in $G$ to at most one vertex in $V(C)$.
Since $G$ is connected, there exists a vertex $x'$ in $V(G)-V(C)$ adjacent in $G$ to a vertex in $V(C)$.
By symmetry, we may assume that $x'$ is adjacent to $v_0$.
Since $G$ is $3$-connected and $C$ is an induced cycle in $G$, $v_{s-1}$ is adjacent in $G$ to a vertex $y'$ in $V(G)-V(C)$.
Since $\lvert V(C) \rvert \geq 5$, $v_0 \neq v_{s-1}$, so $x \neq y$.
Since every vertex in $G-V(C)$ has degree at least $d-1$, by Theorem \ref{mainthm}, there exist $d-2$ admissible paths $P'_1,P'_2,...,P'_{d-2}$ in $G-V(C)$ from $x'$ to $y'$.
Let $Q'_1,Q'_2$ be the subpaths of $C$ with ends $v_{0}$ and $v_{s-1}$ of length $s-1$ and $s+2$, respectively.
Let ${\mathcal C}'=\{(P_i' \cup Q'_j)+xv_{0}+yv_{s-1}: 1 \leq i \leq d-2, 1 \leq j \leq 2\}$.
Since $d-2 \geq 3$, ${\mathcal C}'$ contains $\min\{d-2+3,2(d-2)-2\} \geq d$ cycles of consecutive lengths by Lemma \ref{plus3}, a contradiction.

Hence $G-V(C)$ is not $2$-connected.
It suffices to prove $d \in \{6,7\}$.
Suppose to the contrary that $d \geq 8$.

Suppose that there exists an end-block $B$ of $G-V(C)$ and a vertex $x \in V(B)-\{b_B\}$ adjacent in $G$ to $v_{i_x}$ in $V(C)$ for some $0 \leq i_x \leq 2s$, such that $v_{i_x+s-1}$ is adjacent in $G$ to a vertex $y \in V(G)-(V(C) \cup (V(B)-\{b_B\}))$, where the indices are computed in ${\mathbb Z}_{2s+1}$.
Since $(B,x,b_B)$ is a $2$-connected rooted graph of minimum degree at least $d-2$, Theorem \ref{mainthm} implies that there exist $d-3$ admissible paths in $B$ from $x$ to $b_B$, and hence by concatenating each of them with a fixed path in $G-(V(C) \cup (V(B)-\{b_B\}))$ from $b_B$ to $y$, there exist $d-3$ admissible paths $P''_1,P''_2,...,P''_{d-3}$ in $G-V(C)$ from $x$ to $y$.
Let $Q_1'',Q_2''$ be the subpaths of $C$ with ends $v_{i_x}$ and $v_{i_x+s-1}$ of length $s-1$ and $s+2$, respectively.
Let ${\mathcal C}''=\{(P''_i \cup Q''_j)+xv_{i_x}+yv_{i_x+s-1}: 1 \leq i \leq d-3, 1 \leq j \leq 2\}$.
Since $d-3 \geq 3$, by Lemma \ref{plus3}, ${\mathcal C}''$ contains $\min\{(d-3)+3, 2(d-3)-2\} = \min\{d,2d-8\}$ cycles of consecutive lengths.
Since $G$ does not contain $d$ cycles of consecutive lengths, $2d-8<d$.
Hence $d \in \{6,7\}$, a contradiction.

Hence for every end-block $B$ of $G-V(C)$ and every vertex $x \in V(B)-\{b_B\}$ adjacent in $G$ to $v_{i_x}$ for some $0 \leq i_x \leq 2s$, $N_G(v_{i_x+s-1}) \subseteq V(B)-\{b\}$, where the indices are computed in ${\mathbb Z}_{2s+1}$.
Similarly, for every end-block $B$ of $G-V(C)$ and every vertex $x \in V(B)-\{b_B\}$ adjacent in $G$ to $v_{i_x}$ for some $0 \leq i_x \leq 2s$, $N_G(v_{i_x-(s-1)}) \subseteq V(B)-\{b\}$, where the indices are computed in ${\mathbb Z}_{2s+1}$.

For every end-block $B$ of $G-V(C)$, let $S_B=\{i: 0 \leq i \leq 2s,v_i \in N_G(B-b_B)\}$.
Note that for every end-block $B$ and $i \in {\mathbb Z}_{2s+1}$, if $i \in S_B$, then $\{i+(s-1),i-(s-1)\} \subseteq S_B-S_{B'}$ for every end-block $B'$ of $G-V(C)$ other than $B$.
So if $s-1$ is relatively prime to $2s+1$, then $S_B = \{i: 0 \leq i \leq 2s\}$ for every end-block $B$, but there are at least two end-blocks of $G-V(C)$, a contradiction.

Hence $s-1$ is not relatively prime to $2s+1$.
So there exists a prime $p$ that divides $s-1$ and $2s+1$.
Hence $p$ divides $(2s+1)-2(s-1)=3$.
That is, $p=3$, and $3$ is the greatest common divisor of $2s+1$ and $s-1$.
So for every end-block $B$ and $i \in {\mathbb Z}$, if $i \in S_B$, then since $S_B$ contains $i+t(s-1)$ for every integer $t$, where the computation is in ${\mathbb Z}_{2s+1}$, $S_B$ contains $i+3t'$ for every integer $t'$.
Hence for every $i \in \{0,1,2\}$, either $S_B \supseteq \{i+3t: t \in {\mathbb Z}\}$ or $S_B \cap \{i+3t: t \in {\mathbb Z}\} = \emptyset$, where the computation is in ${\mathbb Z}_{2s+1}$.
Since there are at least two end-blocks in $G-V(C)$, there exists an end-block $B^*$ such that there uniquely exists $i^* \in \{0,1,2\}$ such that $S_{B^*} \cap \{i^*+3t: t \in {\mathbb Z}\} \neq \emptyset$.
This implies that every vertex in $B^*-b_{B^*}$ is adjacent in $G$ to at most one vertex in $V(C)$.

By symmetry, we may assume that $i^*=0$, and there exist $x^*,y^* \in V(B^*)-\{b_{B^*}\}$ such that $x^*v_0 \in E(G)$ and $y^*v_{s-1} \in E(G)$.
Since $G$ is $3$-connected, $x^*$ and $y^*$ can be chosen to be distinct vertices.
Hence $x^*,y^*,b_{B^*}$ are distinct vertices.
Since $B$ is $2$-connected and every vertex in $B^*$ other than $b_{B^*}$ is of degree at least $d-1$ in $B^*$,
by Lemma \ref{three vertices} there exist $d-3$ admissible paths $P_1^*,P_2^*,...,P_{d-3}^*$ in $B^*$ from $x^*$ to $y^*$.
Let $Q_1^*,Q_2^*$ be the subpaths of $C$ with ends $v_0$ and $v_{s-1}$ of length $s-1$ and $s+2$, respectively.
Let ${\mathcal C}^* = \{(P_i^* \cup Q_j^*)+v_0x^*+v_{s-1}y^*: 1 \leq i \leq d-3, 1 \leq j \leq 2\}$.
Since $d-3 \geq 3$, by Lemma \ref{plus3}, ${\mathcal C}^*$ contains $\min\{d-3+3, 2(d-3)-2\}$ cycles of consecutive lengths.
Since $G$ does not contain $d$ cycles of consecutive lengths, $2d-8<d$.
Hence $d \in \{6,7\}$, a contradiction.
This proves the lemma.
\end{proof}

\begin{lem} \label{3connnb}
Let $d$ be an integer with $d \geq 6$.
If $G$ is a $3$-connected non-bipartite graph with minimum degree at least $d$
that does not contain a $K_3$ subgraph, then $G$ contains $d$ admissible cycles.
Furthermore, if $d \geq 8$, then $G$ contains $d$ cycles of consecutive lengths;
and if $d \in \{6,7\}$, then $G$ contains cycles of all lengths modulo $d$.
\end{lem}

\begin{proof}
We may assume that $G$ does not contain $d$ cycles of consecutive lengths, for otherwise we are done.
By Lemma \ref{no_K3_nb1}, $d \in \{6,7\}$ and there exists an induced cycle $C$ in $G$, denoted by $v_0v_1v_2...v_{2s}v_0$, of length at least five such that
	\begin{itemize}
		\item $G-V(C)$ is connected but not $2$-connected,
		\item every end-block of $G-V(C)$ is $2$-connected, and
		\item for every non-cut vertex $v$ of $G-V(C)$, $\lvert N_G(v) \cap V(C) \rvert \leq 2$, and if $\lvert N_G(v) \cap V(C) \rvert = 2$, then $N_G(v) \cap V(C) =\{v_i,v_{i+2}\}$ for some $i \in {\mathbb Z}_{2s+1}$ and the indices are computed in ${\mathbb Z}_{2s+1}$.
	\end{itemize}
Since $d \in \{6,7\}$, to prove this lemma, it suffices to prove that $G$ contains $d$ admissible cycles, and prove that $G$ contains cycles of all lengths modulo $d$.
Note that if $d=7$ and $G$ contains $d$ admissible cycles, then $G$ contains cycles of all lengths modulo $d$.

Suppose to the contrary that either $G$ does not contain $d$ admissible cycles, or $d=6$ and $G$ does not contain cycles of all lengths modulo $d$.

Since $G-V(C)$ is not $2$-connected and every end-block of $G-V(C)$ is $2$-connected, for every end-block $B$ of $G-V(C)$, there exists exactly one vertex $b_B$ in $B$ such that $b_B$ is a cut-vertex of $G-V(C)$.
For every end-block $B$ of $G-V(C)$, let $u_B$ be a vertex in $B-\{b_B\}$ such that $\lvert N_G(u_B) \cap V(C) \rvert$ is as large as possible.
Note that for every end-block $B$ of $G-V(C)$, $(B,u_B,b_B)$ is a $2$-connected rooted graph of minimum degree at least $d-\lvert N_G(u_B) \cap V(C) \rvert$, so there exist $d-\lvert N_G(u_B) \cap V(C) \rvert-1$ admissible paths $P_{B,1},P_{B,2},...,P_{B,d-\lvert N_G(u_B) \cap V(C) \rvert-1}$ in $B$ from $u_B$ to $b_B$.
In addition, for every end-block $B$ of $G-V(C)$, $1 \leq \lvert N_G(u_B) \cap V(C) \rvert \leq 2$ since $u_B \in V(B)-\{b_B\}$.

Suppose that there exists an end-block $B_1$ of $G-V(C)$ such that $\lvert N_G(u_{B_1}) \cap V(C) \rvert = 1$.
Let $B_2$ be an end-block of $G-V(C)$ other than $B_1$.
Let $x \in N_G(u_{B_1}) \cap V(C)$.
Since $G$ is $3$-connected, $u_{B_2}$ can be chosen such that $N_G(u_{B_2}) \cap V(C) - \{x\} \neq \emptyset$.
Let $y \in N_G(u_{B_2}) \cap V(C)-\{x\}$.
Let $Q$ be a path in $C$ from $x$ to $y$.
Let $R$ be a path in $G-V(C)$ from $b_{B_1}$ to $b_{B_2}$.
Let ${\mathcal C} = \{(P_{B_1,i} \cup R \cup P_{B_2,j} \cup Q)+xu_{B_1}+yu_{B_2}: 1 \leq i \leq d-\lvert N_G(u_{B_1}) \cap V(C) \rvert-1, 1 \leq j \leq d-\lvert N_G(u_{B_2}) \cap V(C) \rvert-1\}$.
So ${\mathcal C}$ contains $(d-\lvert N_G(u_{B_1}) \cap V(C) \rvert-1)+ (d-\lvert N_G(u_{B_2}) \cap V(C) \rvert-1)-1 = 2d-4-\lvert N_G(u_{B_2}) \cap V(C) \rvert \geq 2d-6 \geq d$ admissible cycles.
Hence $d=6$ and $G$ does not contain $d$ cycles of consecutive lengths.
So the lengths of the cycles in ${\mathcal C}$ form an arithmetic progression of common difference two.
It follows that the set $\{P_{B_1,i}\cup R \cup P_{B_2,j}: 1\leq i\leq d-|N_G(u_{B_1})\cap V(C)|-1, 1\leq j\leq d-|N_G(u_{B_2})\cap V(C)|-1\}$ contains $2d-6$ paths whose of lengths form an arithmetic progression of common difference two from $u_{B_1}$ to $u_{B_2}$ in $G-V(C)$.
Let $Q_o$ be the odd path from $x$ to $y$ in $C$ and $Q_e$ be the even path from $x$ to $y$ in $C$.
By concatenating each of $Q_o$ and $Q_e$ with $P_{B_1,i}\cup R \cup P_{B_2,j} \cup xu_{B_1} \cup yu_{B_2}$, we could obtain $2d-6$ cycles of consecutive odd lengths and $2d-6$ cycles of consecutive even lengths.
Since $d=6$ is even, $G$ contains cycles of all lengths modulo $d$.

Hence for every end-block $B$ of $G-V(C)$, $\lvert N_G(u_B) \cap V(C) \rvert=2$.
Let $B_3,B_4$ be two distinct end-blocks of $G-V(C)$.
By symmetry, we may assume that $N_G(u_{B_3}) \cap V(C) = \{v_0,v_2\}$.
Since $G$ does not contain a $K_3$ subgraph,
$\lvert N_G(u_{B_4}) \cap V(C) \cap \{v_0,v_1\} \rvert\leq 1$, so there exists $z \in N_G(u_{B_4}) \cap V(C)-\{v_0,v_1\}$.
Let $Q_1$ be the path in $C$ from $v_0$ to $z$ containing $v_0v_1v_2$, and let $Q_2$ be the subpath of $Q_1$ from $v_2$ to $z$.
Note that $\lvert E(Q_1) \rvert = \lvert E(Q_2) \rvert+2$.
Let $R'$ be a path in $G-V(C)$ from $b_{B_3}$ to $b_{B_4}$.
Let ${\mathcal C}' = \{(P_{B_3,i} \cup R' \cup P_{B_4,j} \cup Q_1)+u_{B_3}v_0+u_{B_4}z, (P_{B_3,i} \cup R' \cup P_{B_4,j} \cup Q_2)+u_{B_3}v_2+u_{B_4}z: 1 \leq i \leq d-\lvert N_G(u_{B_3}) \cap V(C) \rvert-1, 1 \leq j \leq d-\lvert N_G(u_{B_4}) \cap V(C) \rvert-1\}$.
Since $\lvert E(Q_1) \rvert = \lvert E(Q_2) \rvert+2$, ${\mathcal C}'$ contains $(d-\lvert N_G(u_{B_3}) \cap V(C) \rvert-1) + (d-\lvert N_G(u_{B_4}) \cap V(C) \rvert-1)-1 + 2-1 = 2d-6 \geq d$ admissible cycles.
So $d=6$ and $G$ does not contain $d$ cycles of consecutive lengths.
Hence the lengths of these cycles form an arithmetic progression of common difference two.
It follows that $P_{B_3,i}\cup R' \cup P_{B_4,j}$ for all $1\leq i\leq d-|N_G(u_{B_3})\cap V(C)|-1$ and $1\leq j\leq d-|N_G(u_{B_4})\cap V(C)|-1$ contains $2d-7$ paths whose of lengths form an arithmetic progression of common difference two between $u_{B_3}$ to $u_{B_4}$ in $G-V(C)$.
Let $Q_o'$ and $Q_e'$ be the odd path and even path in $C$ from $z$ to $v_0$, respectively.
By concatenating each of $Q_o'$ and $Q_e'$ with $P_{B_3,i}\cup R' \cup P_{B_4,j}\cup u_{B_3}v_0 \cup u_{B_4}z$,
we obtain $2d-7$ cycles of consecutive odd lengths and $2d-7$ cycles of consecutive even lengths.
Since $d=6$ is even, $G$ contains cycles of all lengths modulo $d$, a contradiction.
This proves the lemma.
\end{proof}

Recall that for every positive integer $d$, $K_{d,d}^-$ is the graph obtained from $K_{d,d}$ by deleting an edge.

\begin{lem}\label{deanc3}
Let $d \geq 3$ be an integer.
Let $G$ be a $\max\{d,5\}$-connected graph of girth exactly four and of minimum degree at least $d$ that does not contain a cycle of length five.
If $G$ does not contain a $K_{d,d}^-$ subgraph, then $G$ contains $d$ admissible cycles.
\end{lem}
\begin{proof}
Suppose to the contrary that $G$ does not contain $d$ admissible cycles.
	
Since the girth of $G$ equals four, $G$ contains a $K_{2,2}$ subgraph.
So there exists a complete bipartite subgraph $Q$ of $G$ with bipartition $(Q_1,Q_2)$ such that
	\begin{itemize}
		\item[(i)] $2\leq|Q_1|\leq|Q_2|$,
		\item[(ii)] subject to (i), $|Q_1|$ is maximum, and
		\item[(iii)] subject to (i) and (ii), $|Q_2|$ is maximum.
	\end{itemize}
Let $s=|Q_1|$.
Note that $2\leq s \leq d-1$ since $G$ does not contain a $K_{d,d}$ subgraph.
Since $G$ does not contain a $K_3$ subgraph, $Q$ is an induced subgraph of $G$, and for every vertex $v$ of $G-V(Q)$, either $N_G(v)
\cap Q_1=\emptyset$ or $N_G(v) \cap Q_2 = \emptyset$.

For any $v\in V(G)-V(Q)$, $|N_G(v)\cap Q_1|\leq s-1$ by (iii), and $|N_G(v)\cap Q_2|\leq s$ by (ii).
If there exists a vertex $z \in V(G)-V(Q)$ such that $z$ is adjacent in $G$ to at least $s$ vertices in $V(Q)$, then let $Z=\{z\}$;
otherwise, let $Z$ be the empty set.
Note that if $Z \neq \emptyset$, then $N_G(z) \cap V(Q) \subseteq Q_2$ and $\lvert Q_2 \rvert \geq s+1$, since $G$ is of girth four
and by (i)-(iii).

Suppose there exists $t \in \{1,2\}$ such that there exists a component $M$ of $G-(Q_t \cup Z)$ disjoint from $Q_{3-t}$.
Since $G$ is $d$-connected, $\lvert Q_t \rvert + \lvert Z \rvert \geq d$.
If $\lvert Q_t \rvert=s$, then since $d \geq s+1$, $Z \neq \emptyset$ and $\lvert Q_t \rvert = s=d-1$, so $t=1$ and $\lvert Q_{3-t}
\rvert \geq s+1=d$, and hence $G[V(Q) \cup Z]$ contains a $K_{d,d}^-$ subgraph, a contradiction.
Hence $\lvert Q_t \rvert \geq s+1$.
In particular, $t=2$.
Note that when $s=2$, $\lvert Q_t \rvert \geq s+2$, for otherwise $\lvert Q_t \rvert+\lvert Z \rvert \leq (s+1)+1=4$, contradicting
that $G$ is 5-connected.
Since $G-Z$ is $4$-connected, we have that $|N_G(M)\cap Q_t|\geq 4$. If $s=2$, let $A$ be a subset of $N_G(M)\cap Q_t-N_{G}(z)$
with size two; if $s \geq 3$,  let $A$ be a subset of $N_G(M)\cap Q_t$ with size two. Note that $N_G(M) \cap Q_t- A \neq \emptyset.$
If $\lvert Q_t \rvert=s+1$ and $Z \neq \emptyset$, then let $\epsilon=1$; for otherwise, let $\epsilon=0$.
Let $G_M$ be the graph obtained from $G[V(M) \cup Q_t]$ by identifying all vertices in $A$ into a vertex $u_M$, identifying all
vertices in $Q_t-A$ into a vertex $v_M$, and deleting all resulting loops and parallel edges.
Since $\lvert Q_t \rvert \geq s+1$ and $t=2$, and since $G$ is of girth four and does not contain a 5-cycle, no vertex of $M$ is
adjacent in $G$ to both $Z$ and $Q_2$, so the minimum degree of $(G_M,u_M,v_M)$ is at least $d-(s-2)-\lvert Z \rvert+\epsilon$ by the
definition of $Z$.
So $(G_M,u_M,v_M)$ is a rooted graph of minimum degree at least $d-(s-2)-\lvert Z \rvert+\epsilon$.
Since $G-Z$ is $3$-connected and $M$ is a component of $G-(Q_t \cup Z)$, we know $(G_M,u_M,v_M)$ is a $2$-connected rooted graph of
minimum degree at least $d-s+2-\lvert Z \rvert+\epsilon$.
By Theorem \ref{mainthm}, there exist $d-s+1-\lvert Z \rvert+\epsilon$ admissible paths in $G_M$ from $u_M$ to $v_M$.
So there exist $d-s+1-\lvert Z \rvert+\epsilon$ admissible paths $P_{M,1},P_{M,2},...,P_{M,d-s+1-\lvert Z \rvert+\epsilon}$ in $G[V(M)
\cup Q_t]$ from $A$ to $Q_t-A$ internally disjoint from $V(Q) \cup Z$.
For each $i$ with $1 \leq i \leq d-s+1-\lvert Z \rvert+\epsilon$, let $\alpha_i$ be the ends of $P_{M,i}$ in $A$ and let $\beta_i$ be
the end of $P_{M,i}$ in $Q_t-A$.
Since $\lvert Q_t \rvert \geq s+1$ and $t=2$ and $Q$ is a complete bipartite graph, there exist $s+\lvert Z \rvert-\epsilon$
admissible paths $Q_{M,1},...,Q_{M,s+\lvert Z \rvert-\epsilon}$ in $G[V(Q) \cup Z]$ from $\alpha_i$ to $\beta_i$.
Then the set $\{P_{M,j} \cup Q_{M,k}: 1 \leq j \leq d-s+1-\lvert Z \rvert+\epsilon, 1 \leq k \leq s+\lvert Z \rvert-\epsilon\}$
contains $(d-s+1-\lvert Z \rvert+\epsilon)+(s+\lvert Z \rvert-\epsilon)-1=d$ admissible cycles, a contradiction.

So for every $t \in \{1,2\}$, every component of $G-(Q_t \cup Z)$ intersects $Q_{3-t}$.
Let $G'$ be the graph obtained from $G-Z$ by identifying all vertices in $Q_1$ into a vertex $u'$, identifying all vertices in $Q_2$
into a vertex $v'$, and deleting resulting loops and parallel edges.
Since $G$ is of girth four and does not contain a 5-cycle, no vertex of $G-(V(Q) \cup Z)$ is adjacent in $G$ to either both $Z$ and
$Q_2$ or both $Q_1$ and $Q_2$, so the minimum degree of $(G',u',v')$ is at least $d-(s-2)-\lvert Z \rvert$ by the definition of $Z$.
Since $G$ is $3$-connected, $G-Z$ is $2$-connected, so every cut-vertex of $G'$ is $u'$ or $v'$.
Since for every $t \in \{1,2\}$, every component of $G-(Q_t \cup Z)$ intersects $Q_{3-t}$, we know $G'$ is $2$-connected.
So $(G',u',v')$ is a $2$-connected rooted graph of minimum degree at least $d-s-\lvert Z \rvert+2$.
By Theorem \ref{mainthm}, there exist $d-s-\lvert Z \rvert+1$ admissible paths in $G'$ from $u'$ to $v'$.
So there exist $d-s-\lvert Z \rvert+1$ admissible paths $R_1,R_2,...,R_{d-s-\lvert Z \rvert+1}$ in $G-Z$ from $Q_1$ to $Q_2$
internally disjoint from $V(Q)$.
For every $i$ with $1 \leq i \leq d-s-\lvert Z \rvert+1$, let $x_i,y_i$ be the ends of $R_i$ in $Q_1,Q_2$, respectively.
Since $Q$ is a complete bipartite graph, for each $i$ with $1 \leq i \leq d-s-\lvert Z \rvert+1$, $G[V(Q) \cup Z]$ contains $s+\lvert
Z \rvert$ admissible paths $R_1',R_2',...,R_{s+\lvert Z \rvert}'$ from $x_i$ to $y_i$.
So the set $\{R_j \cup R'_k: 1 \leq j \leq d-s-\lvert Z \rvert+1, 1 \leq k \leq s+\lvert Z \rvert\}$ contains $d$ admissible cycles, a
contradiction.
This proves the lemma.
\end{proof}

\begin{lem} \label{binonsep}
Let $G$ be a $3$-connected bipartite graph.
If $G$ does not contain a cycle of length four, then $G$ contains a non-separating induced cycle $C$ such that for every non-cut-vertex $v$ of $G-V(C)$, $|N_G(v)\cap V(C)|\leq1$.
\end{lem}

\begin{proof}
Let $C$ be a cycle of $G$ such that
	\begin{itemize}
		\item[(i)] the largest component of $G-V(C)$ is as large as possible, and
		\item[(ii)] subject to (i), $\lvert V(C) \rvert$ is as small as possible.
	\end{itemize}
Let $M$ be a component of $G-V(C)$ with $\lvert V(M) \rvert$ maximum.

If there exists a chord $e$ of $C$, then one of $P_e+e$ and $Q_e+e$ is a cycle shorter than $C$ such that $M$ is a component of the graph obtained from $G$ by deleting this cycle, a contradiction, where $P_e,Q_e$ are the two subpaths of $C$ with ends the same as $e$.
Hence $C$ is an induced cycle.

Suppose there exists a component $M'$ of $G-V(C)$ other than $M$.
Let $A=N_G(M)\cap V(C)$ and $B=N_G(M')\cap V(C)$.
Since $G$ is $3$-connected, $\min\{|A|,|B|\} \geq 3$.
Since $\lvert A \rvert \geq 3$ and $\lvert B \rvert \geq 2$, there exists a subpath $Q$ of $C$ whose ends belong to $B$ such that some internal vertex of $Q$ belongs to $A$.
Since $M'$ is connected, there exists a path $Q'$ from one end of $Q$ to another end of $Q$ such that all internal vertices belong to $V(M')$.
Let $Q''$ be the subpath of $C$ with the same ends as $Q$ but internally disjoint from $Q$.
Then $Q' \cup Q''$ is a cycle in $G$ such that some component of $G-V(Q' \cup Q'')$ contains $M$ and a vertex in $A$, contradicting (i).

Hence $C$ is a non-separating cycle in $G$.
Suppose that there exists a non-cut-vertex $v$ of $G-V(C)$ such that $\lvert N_G(v) \cap V(C) \rvert \geq 2$.
Let $x,y$ be distinct vertices in $N_G(v) \cap V(C)$ such that no internal vertex of $R_1$ belongs to $N_G(v) \cap V(C)$, where $R_1,R_2$ are the two subpaths of $C$ with ends $x$ and $y$.
If $\lvert E(R_1) \rvert \leq 2$, then $R_1+vx+vy$ is a cycle of length at most four, contradicting that $G$ is a bipartite graph with no 4-cycle.
So $\lvert E(R_1) \rvert \geq 3$.
Hence $R_2+vx+vy$ is a cycle shorter than $C$.
Since $\lvert E(R_1) \rvert \geq 3$, there exist distinct internal vertices $x',y'$ of $R_1$.
Since $C$ is an induced cycle and every vertex of $G$ has degree at least three, $N_G(x')-V(C) \neq \emptyset \neq N_G(y')-V(C)$.
Since $C$ is a non-separating cycle, $N_G(x') \cap V(M) = N_G(x')-V(C) \neq \emptyset \neq N_G(y')-V(C) = N_G(y') \cap V(M)$.
Since $x',y'$ are internal vertices of $R_1$, $\{x',y'\} \cap N_G(v) = \emptyset$.
So $N_G(x') \cap V(M)-\{v\} \neq \emptyset$ and $N_G(y') \cap V(M)-\{v\} \neq \emptyset$.
Since $v$ is not a cut-vertex of $G-V(C)$, $M-v$ is connected.
So some component of $G-V(R_2+vx+vy)$ contains $(V(M)-\{v\}) \cup \{x',y'\}$, contradicting (i).
This proves the lemma.
\end{proof}

\begin{lem}\label{deanbi}
Let $d\geq5$ be an integer.
Let $G$ be a $3$-connected bipartite graph of minimum degree at least $d$.
If $G$ does not contain cycles of length four, then $G$ contains $d$ admissible cycles.
\end{lem}

\begin{proof}
Suppose to the contrary that $G$ does not contain $d$ admissible cycles.
By Lemma \ref{binonsep}, there exists a positive integer $s$ and an induced non-separating cycle $C=v_0v_1\ldots v_{2s-1}v_0$ in $G$ such that for every non-cut-vertex of $G-V(C)$, it is adjacent in $G$ to at most one vertex in $V(C)$.
Since $G$ is a bipartite graph with no 4-cycle, $s\geq3$.
For any any $0 \leq i <j \leq 2s-1$, let $Q_{i,j}$ and $Q_{i,j}'$ be the two subpaths of $C$ with ends $v_i,v_j$.

Suppose $G-V(C)$ is $2$-connected.
Since $C$ is an induced non-separating cycle and $G$ is of minimum degree at least $d \geq 4$, there exist distinct vertices $x,y$ in $V(G)-V(C)$ such that $\{xv_0,yv_{s-2}\} \in E(G)$.
Since $(G-V(C),x,y)$ is a $2$-connected rooted graph of minimum degree at least $d-1$, there exist $d-2$ admissible paths $P_1,P_2,...,P_{d-2}$ in $G-V(C)$ from $x$ to $y$ by Theorem \ref{mainthm}.
Note that $\lvert \lvert E(Q_{0,s-2}) \rvert - \lvert E(Q_{0,s-2}') \rvert \rvert = 4$.
Since $d-2>2$ and $G$ is bipartite, the set $\{(P_i \cup Q_{0,s-2})+xv_0+yv_{s-2}, (P_i \cup Q_{0,s-2}')+xv_0+yv_{s-2}: 1 \leq i \leq d-2\}$ contains $d$ admissible cycles, a contradiction.

So $G-V(C)$ is not $2$-connected.
In particular, there exist two distinct end-blocks $B_1,B_2$ of $G-V(C)$.
Since $G$ is $3$-connected, each $B_1,B_2$ is $2$-connected.
For $i \in \{1,2\}$, let $b_i$ be the cut-vertex of $G-V(C)$ contained in $V(B_i)$.
Since $G$ is $2$-connected, for each $i \in \{1,2\}$, there exists an integer $r_i$ with $0 \leq r_i \leq 2s-1$ and a vertex $u_i$ in $V(B_i)-\{b_i\}$ such that $u_iv_{r_i} \in E(G)$.
Since $G$ is $3$-connected, $r_1$ and $r_2$ can be chosen to be distinct.
For $i \in \{1,2\}$, since $(B_i,u_i,b_i)$ is a $2$-connected rooted graph of minimum degree at least $d-1$, there exist $d-2$ admissible paths $P_{i,1},P_{i,2},...,P_{i,d-2}$ in $B_i$ from $u_i$ to $b_i$.
Let $Q$ be a path in $G-V(C)$ from $b_1$ to $b_2$ internally disjoint from $V(B_1) \cup V(B_2)$.
Then the set $\{(P_{1,i} \cup Q \cup P_{2,j} \cup Q_{r_1,r_2})+u_1v_{r_1}+u_2v_{r_2}: 1 \leq i \leq d-2, 1 \leq j \leq d-2\}$ contains $2(d-2)-1=2d-5 \geq d$ admissible cycles, a contradiction.
This proves the lemma.
\end{proof}

For every graph $H$, a {\it 1-subdivision} of $H$ is a graph that is obtained from $H$ by subdividing each edge exactly once.

\begin{lem} \label{1subdivK4}
	Let $G$ be a graph of girth at least five.
	Let $H$ be a subgraph of $G$ that is a 1-subdivision of $K_4$.
	If there exists a vertex in $V(G)-V(H)$ adjacent in $G$ to two vertices in $V(H)$, then $G$ contains a cycle of length five or ten.
\end{lem}

\begin{proof}
	We may assume $G$ is of girth at least six, for otherwise we are done.
	Let $v$ be a vertex in $V(G)-V(H)$ adjacent in $G$ to two vertices $x,y$ in $V(H)$.
	Let $S$ be the set of vertices of $H$ of degree three.
	Since $G$ has girth at least five, at least one of $x,y$ does not belong to $S$.
	Then since $G$ has girth at least six, both $x,y$ do not belong to $S$.
	So there exist edges $e,e'$ of $K_4$ such that $x$ and $y$ are obtained by subdividing $e$ and $e'$, respectively.
	Since $G$ has girth at least five, $e$ and $e'$ form a matching in $K_4$.
	Let $z$ be a vertex of $H$ obtained by subdividing an edge other than $e,e'$.
	Then $(H-\{z\})+vx+vy$ has a Hamiltonian cycle of length ten.
	This proves the lemma.
\end{proof}

We say a graph is a {\it theta graph} is a subdivision of $K_4^-$.
The {\it branch vertices} of a theta graph are the vertices of degree at least three. 
A subgraph $H$ of a graph $G$ is {\it spanning} if $V(H)=V(G)$.

\begin{lem} \label{mintheta}
Let $G$ be a graph of girth at least six that does not contain a cycle of length ten.
Let $H$ be a subgraph of $G$ isomorphic to a theta graph such that $\lvert V(H) \rvert$ is minimum.
Then the following hold.
	\begin{enumerate}
		\item $H$ is an induced subgraph of $G$.
		\item There exists at most one vertex of $G-V(H)$ adjacent in $G$ to at least two vertices in $V(H)$.
		\item If there exists a vertex $v$ of $G-V(H)$ adjacent in $G$ to at least two vertices in $V(H)$, then $G[V(H) \cup \{v\}]$ is isomorphic to a 1-subdivision of $K_4$.
	\end{enumerate}
\end{lem}

\begin{proof}
Suppose that $H$ is not induced.
Then there exists $e \in E(G)-E(H)$ with both ends in $V(H)$.
Since the girth of $G$ is at least six, there exists a subgraph $H'$ of $G$ with $V(H') \subset V(H)$ such that $H'$ is a theta graph, contradicting the minimality of $\lvert V(H) \rvert$.

So $H$ is an induced subgraph.
We may assume there exists a vertex $v$ of $G-V(H)$ adjacent in $G$ to at least two vertices in $V(H)$, for otherwise we are done.
Let $x,y$ be the branch vertices of $H$.
Let $P_1,P_2,P_3$ be the three internally disjoint paths in $H$ from $x$ to $y$.

By the minimality of $\lvert V(H) \rvert$ and the girth condition of $G$, $v$ is not adjacent to any branch vertices of $H$.
Similarly, for each $i \in \{1,2,3\}$, $v$ is adjacent to at most one vertex in $V(P_i)$.
So there exist distinct $i,j$ such that $v$ is adjacent to exactly one vertex $a$ in $V(P_i)-\{x,y\}$ and exactly one vertex $b$ in $V(P_j)-\{x,y\}$.
By symmetry, we may assume $i=1$ and $j=2$.
Since $(H-(V(P_3)-\{x,y\}))+av+bv$ is a theta graph, by the minimality of $\lvert V(H) \rvert$, $\lvert V(P_3) \rvert \leq 3$.
Let $L_1$ be the subpath of $P_1$ from $x$ to $a$.
Since the graph obtained from $H+av+bv$ by deleting all internal vertices of $L_1$ is a theta graph, $L_1$ contains at most one internal vertex by the minimality of $\lvert V(H) \rvert$.
Similarly, the subpath $L_2$ of $P_2$ from $x$ to $b$ contains at most one internal vertex.
Since $L_1 \cup L_2 \cup avb$ is a cycle in $G$ and $G$ has girth at least six, both $L_1,L_2$ contain exactly one internal vertex.
Similarly, each of the subpath of $P_1$ from $a$ to $y$ and the subpath of $P_2$ from $b$ to $y$ contains exactly one vertex.
Since $P_1 \cup P_3$ is a cycle in $G$, $P_3$ contains exactly one internal vertex.
Hence $G[V(H) \cup \{v\}]$ contains a 1-subdivision of $K_4$ as a spanning subgraph.
Since $G$ is of girth at least six, $G[V(H) \cup \{v\}]$ is isomorphic to a 1-subdivision of $K_4$.

If there exists a vertex $v'$ of $V(G)-V(H)$ other than $v$ adjacent in $G$ to at least two vertices in $V(H)$, then $v'$ is adjacent in $G$ to at least two vertices in $V(H) \cup \{v\}$ which induces a subgraph isomorphic to a 1-subdivision of $K_4$, so $G$ contains a cycle of length ten by Lemma \ref{1subdivK4}, a contradiction.
So $v$ is the only vertex that is adjacent in $G$ to at least two vertices in $V(H)$.
This proves the lemma.
\end{proof}

\begin{lem} \label{2nbK4}
Let $d \geq 5$ be an integer.
Let $G$ be a 5-connected graph of girth at least six and of minimum degree at least $d$ that does not contain a cycle of length ten.
Let $H$ be an induced subgraph of $G$ isomorphic to a 1-subdivision of $K_4$.
Then $G$ contains $d$ admissible cycles.
\end{lem}

\begin{proof}
Suppose to the contrary that $G$ does not contain $d$ admissible cycles.
Note that every vertex of $G-V(H)$ is adjacent in $G$ to at most one vertex in $V(H)$ by Lemma \ref{1subdivK4}.
We say a pair of two distinct vertices $x,y$ of $H$ are {\it useful} if there exist paths $H_1,H_2,H_3$ in $H$ from $x$ to $y$ of lengths $h_1,h_2,h_3$, respectively, such that $(h_1,h_2,h_3) \in \{(1,5,7), (2,4,6), (3,5,7)\}$.

Let $M$ be a component of $G-V(H)$.
Since $G$ is 5-connected, there exists a useful pair of vertices $x,y$ such that $x$ is adjacent in $G$ to some vertex $x'$ in $V(M)$ and $y$ is adjacent in $G$ to some vertex $y'$ in $V(M)$.
Note that $x' \neq y'$, for otherwise some vertex of $M$ is adjacent in $G$ to two vertices in $V(H)$.

Suppose $M$ is $2$-connected.
Then $(M,x',y')$ is a $2$-connected rooted graph of minimum degree at least $d-1$.
By Theorem \ref{mainthm}, there exist $d-2$ admissible paths $P_1,...,P_{d-2}$ in $G'$ from $x'$ to $y'$.
Since $d \geq 5$, the set $\{(P_i \cup H_j)+xx'+yy': 1 \leq i \leq d-2, 1 \leq j \leq 3\}$ contains $d$ admissible cycles, a contradiction.

So $M$ is not $2$-connected.
Let $B_1,B_2$ be two distinct end-blocks of $M$.
Let $b_1,b_2$ be the cut-vertex of $G-V(H)$ contained in $V(B_1),V(B_2)$, respectively.
Since $G$ is $3$-connected, some vertex $x_1$ in $V(B_1)-\{b_1\}$ is adjacent in $G$ to some vertex $u_1$ in $V(H)$, and some vertex $x_2$ in $V(B_2)-\{b_2\}$ is adjacent in $G$ to some vertex $u_2$ in $V(H)$.
For each $i \in \{1,2\}$, since $(B_i, x_i,b_i)$ is a $2$-connected rooted graph of minimum degree at least $d-1$, there exist $d-2$ admissible paths $Q_{i,1},...,Q_{i,d-2}$ in $B_i$ from $x_i$ to $b_i$.
Let $Q$ be a path in $M$ from $b_1$ to $b_2$ internally disjoint from $V(B_1) \cup V(B_2)$, and let $Q'$ be a path in $H$ from $u_1$ to $u_2$.
Then the set $\{(Q_{1,i} \cup Q \cup Q_{2,j} \cup Q')+x_1u_1+x_2u_2: 1 \leq i \leq d-2, 1 \leq j \leq d-2\}$ contains $2(d-2)-1=2d-5 \geq d$ admissible cycles, a contradiction.
This proves the lemma.
\end{proof}

\begin{lem} \label{interface}
Let $H$ be a theta graph.
Then there exist two distinct vertices $x,y$ and three paths in $H$ from $x$ to $y$ such that the lengths of these three paths modulo 5 are pairwise distinct.
\end{lem}

\begin{proof}
Let $P_1,P_2,P_3$ be the three internally disjoint paths in $H$ between the branch vertices of $H$.
For each $i \in \{1,2,3\}$, we denote $P_i$ by $v_{i,0}v_{i,1}...v_{i,\lvert E(P_i) \rvert}$, where $v_{1,0}=v_{2,0}=v_{3,0}$.
For each $i \in \{1,2,3\}$ and each $j \in \{1,...,\lvert E(P_i) \rvert\}$, let $L_{i,j}=v_{i,0}v_{i,1}...v_{i,j}$ and let $R_{i,j}=v_{i,j}v_{i,j+1}...v_{i,\lvert E(P_i) \rvert}$.

Suppose to the contrary that there do not exist two distinct vertices $x,y$ and three paths from $x$ to $y$ with pairwise distinct lengths modulo 5.
So the lengths of $P_1,P_2,P_3$ modulo 5 are not pairwise distinct.
Hence, by symmetry, there exists $t \in \{0,1,2,3,4\}$ such that $\lvert E(P_1) \rvert$ and $\lvert E(P_2) \rvert$ equal $t$ modulo 5.

Suppose that $\lvert E(P_3) \rvert=t$ modulo 5.
By symmetry, we may assume that $\lvert E(P_3) \rvert \leq \lvert E(P_i) \rvert$ for every $i \in \{1,2\}$.
So $\min\{\lvert E(P_1) \rvert,\lvert E(P_2) \rvert\} \geq 2$.
Note that the paths $R_{1,1} \cup R_{2,2}, L_{1,1} \cup P_3 \cup R_{2,2}, R_{1,1} \cup P_3 \cup L_{2,2}$ are three paths from $v_{1,1}$ to $v_{2,2}$ with lengths $2t-3,2t-1,2t+1$ modulo 5, respectively, a contradiction.

Hence there exists $s \in \{0,1,2,3,4\}-\{t\}$ such that $\lvert E(P_3) \rvert =s$ modulo 5.
By symmetry, we may assume that $\lvert E(P_1) \rvert >1$.
For every $r \in \{1,2\}$, the paths $R_{1,r}, L_{1,r} \cup P_2, L_{1,r} \cup P_3$ are three paths from $v_{1,r}$ to $v_{1,\lvert E(P_1) \rvert}$ of lengths $t-r,t+r,s+r$ modulo 5, respectively, so $t-r=s+r$ modulo 5.
That is, $t-1=s+1$ modulo 5 and $t-2=s+2$ modulo 5, a contradiction.
This proves the lemma.
\end{proof}

\begin{lem} \label{3plus3mod5}
Let $a$ and $d$ be integers such that $d \in \{1,2\}$.
Let $B$ be a subset of $\{0,1,2,3,4\}$ of size three.
Then the set $\{a+id+b: 0 \leq i \leq 2, b \in B\}$ contains a multiple of 5.
\end{lem}

\begin{proof}
Let $X=\{a+id+b: 0 \leq i \leq 2, b \in B\}$.
If there exists an integer $s$ such that the three elements of $B$ are either $s,s+1,s+2$ modulo 5 or $s,s+2,s+4$ modulo 5, then $X$ contains a multiple of 5.
So by shifting, we may without loss of generality assume that $B=\{0,1,3\}$.
If $d=1$, then $X \supseteq \{a,a+1,a+2,a+3,a+4\}$, so $X$ contains a multiple of 5.
If $d=2$, then $X \supseteq \{a,a+1,a+2,a+3,a+4\}$, so $X$ contains a multiple of 5.
This proves the lemma.
\end{proof}

\begin{lem}\label{dean5}
If $G$ is a $5$-connected graph of girth at least five, then $G$ contains a cycle of length $0$ modulo $5$.
\end{lem}

\begin{proof} 
Suppose to the contrary that $G$ does not contain a cycle of length 0 modulo 5.
In particular, $G$ does not contain a 5-cycle and a 10-cycle, and $G$ does not contain five admissible cycles.
So the girth of $G$ is at least six, and $G$ does not contain a cycle of length ten.

Let $H$ be a subgraph of $G$ isomorphic to a theta graph with $\lvert V(H) \rvert$ minimum.
By Lemma \ref{mintheta}, $H$ satisfies the following.
	\begin{itemize}
		\item $H$ is an induced subgraph of $G$.
		\item There exists at most one vertex of $G-V(H)$ adjacent in $G$ to at least two vertices in $V(H)$.
		\item If there exists a vertex $v$ of $G-V(H)$ adjacent in $G$ to at least two vertices in $V(H)$, then $G[V(H) \cup \{v\}]$ is isomorphic to a 1-subdivision of $K_4$.
	\end{itemize}

If there exists a vertex of $G-V(H)$ adjacent in $G$ to at least two vertices of $V(H)$, then there exists an induced subgraph $H'$ isomorphic to an induced 1-subdivision of $K_4$, so by Lemma \ref{2nbK4}, $G$ contains five admissible cycles, a contradiction.

So every vertex of $G-V(H)$ is adjacent in $G$ to at most one vertex in $V(H)$.
Let $G'=G-V(H)$.
Let $d=5$.

Suppose that there exists a component $M$ of $G'$ such that $M$ is not $2$-connected.
Let $B_1,B_2$ be distinct end-blocks of $M$.
Since $G$ is $3$-connected and every vertex in $V(M)$ is adjacent in $G$ to at most one vertex in $V(H)$, $B_1$ and $B_2$ are $2$-connected.
For each $i \in \{1,2\}$, let $b_i$ be the cut-vertex of $M$ contained in $V(B_i)$.
Since $G$ is $3$-connected, for each $i \in \{1,2\}$, there exists $x_i \in V(B_i)-\{b_i\}$ such that $x_i$ is adjacent in $G$ to some vertex $y_i$ in $V(H)$.
For each $i \in \{1,2\}$, since $(B_i,x_i,b_i)$ is a $2$-connected rooted graph of minimum degree at least $d-1$, there exist $d-2$ admissible paths $P_{i,1},...,P_{i,d-2}$ in $B_i$ from $x_i$ to $b_i$ by Theorem \ref{mainthm}.
Let $Q$ be a path in $M$ from $b_1$ to $b_2$ internally disjoint from $V(B_1) \cup V(B_2)$.
Let $Q'$ be a path in $H$ from $y_1$ to $y_2$.
Then the set $\{(P_{1,i} \cup Q \cup P_{2,j} \cup Q')+x_1y_1+x_2y_2: 1 \leq i \leq d-2, 1 \leq j \leq d-2\}$ contains $2(d-2)-1 \geq d=5$ admissible paths, a contradiction.

So every component of $G'$ is $2$-connected.
Suppose that $G'$ is not connected.
Let $M_1,M_2$ be two distinct components of $G'$.
For each $i \in \{1,2\}$, since $G$ is 4-connected, there exist distinct vertices $x_{i,1}$ and $x_{i,2}$ in $V(M_i)$ such that $x_{i,1}$ is adjacent in $G$ to a vertex $y_{i,1}$ in $V(H)$ and $x_{i,2}$ is adjacent in $G$ to a vertex $y_{i,2}$ in $V(H)$.
For each $i \in \{1,2\}$, since $(M_i,x_{i,1},x_{i,2})$ is a $2$-connected rooted graph of minimum degree at least $d-1$, there exist $d-2$ admissible paths $R_{i,1},...,R_{i,d-2}$ in $M_i$ from $x_{i,1}$ to $x_{i,2}$ by Theorem \ref{mainthm}.
Since $H$ is $2$-connected, there exist two disjoint paths $Q_1,Q_2$ in $H$ from $\{y_{1,1},y_{1,2}\}$ to $\{y_{2,1},y_{2,2}\}$. 
Then the set $\{(R_{1,i} \cup Q_1 \cup R_{2,j} \cup Q_2)+x_{1,1}y_{1,1}+x_{1,2}y_{1,2}+x_{2,1}y_{2,1}+x_{2,2}y_{2,2}: 1 \leq i \leq d-2, 1 \leq j \leq d-2\}$ contains $2(d-2)-1 \geq 5$ admissible cycles, a contradiction.

So $G'$ is $2$-connected.
By Lemma \ref{interface}, there exist two distinct vertices $x,y$ in $H$ such that there exist three paths $A_1,A_2,A_3$ in $H$ from $x$ to $y$ with pairwise distinct lengths modulo 5.
Since $G$ is 5-connected and $H$ is an induced subgraph, there exist distinct vertices $x',y'$ in $V(G')$ such that $\{xx',yy'\} \subseteq E(G)$.
Since $(G',x',y')$ is a $2$-connected rooted graph of minimum degree at least $d-1=4$, by Theorem \ref{mainthm}, there exist three admissible paths $Z_1,Z_2,Z_3$ in $G'$ from $x'$ to $y'$.
By Lemma \ref{3plus3mod5}, the set $\{(Z_i \cup A_j)+ xx'+yy': 1 \leq i \leq 3, 1 \leq j \leq 3\}$ contains a cycle of length 0 modulo 5, a contradiction.
This proves the lemma.
\end{proof}

Now we are ready to prove Theorem \ref{deantheo}. The following is a restatement of Theorem \ref{deantheo}.

\begin{theo}\label{Thm:Dean2}
For $d\geq 3$, every $d$-connected graph contains a cycle of length zero modulo $d$.
\end{theo}

\begin{proof}
By \cite[Theorem 1]{CS} and \cite[Theorem 1.2]{DLS93}, the theorem is true for $d \in \{3,4\}$.
So we may assume that $d \geq 5$.
Suppose to the contrary that $G$ does not contain a cycle of length 0 modulo $d$.
So $G$ does not contain a $K_d'$ subgraph and does not contain a $K_{d,d}^-$ subgraph.
In addition, $G$ does not contain $d$ cycles of consecutive length, and when $d$ is odd or $G$ is bipartite, $G$ does not contain $d$ admissible cycles.

Since $G$ is $(d-1)$-connected and does not contain a $K_d'$ subgraph, $G$ does not contain a $K_4^-$ subgraph by Lemma \ref{usingK_4-}.
Since $G$ does not contain a $K_4^-$ subgraph, by Lemma \ref{K_3noK_4-}, $G$ does not contain a $K_3$ subgraph.
Since $G$ does not contain a $K_3$ subgraph, by Lemma \ref{3connnb}, either $d=5$ or $G$ is bipartite.
Since $G$ does not contain a $K_3$ subgraph and a $K_{d,d}^-$ subgraph, by Lemma \ref{deanc3}, either $G$ does not contain a 4-cycle, or $G$ contains a cycle of length four and a cycle of length five, or $d$ is even.

Suppose that $G$ does not contain a 4-cycle.
Then $G$ is not bipartite by Lemma \ref{deanbi}.
So $d=5$.
Since $G$ does not contain a $K_3$ subgraph and a 4-cycle, $G$ is of girth at least five.
So $G$ contains a cycle of length 0 modulo $5=d$ by Lemma \ref{dean5}, a contradiction.

So either $G$ contains a 4-cycle and a 5-cycle, or $d$ is even.
Note that either case implies $d \neq 5$.
So $G$ is bipartite, contradicting that $G$ contains a 5-cycle.
This proves the theorem.
\end{proof}

When $d \geq 6$, we can strengthen the conclusion of Theorem \ref{Thm:Dean2}.

\begin{theo} \label{thm:Deanbigd}
For integers $d\geq 6$ and $t$ satisfying $2t\neq 2$ modulo $d$,
every $d$-connected graph contains a cycle of length $2t$ modulo $d$.
\end{theo}

\begin{proof}
Suppose to the contrary that there exist integers $d \geq 6$ and $t$ with $2t\neq 2$ modulo $d$
such that there exists a $d$-connected graph $G$ that does not contain a cycle of length $2t$ modulo $d$.
In particular, $G$ does not contain a $K_{d+1}'$ subgraph and does not contain a $K_{d,d}^-$ subgraph.
In addition, $G$ does not contain $d$ cycles of consecutive length, and when $d$ is odd or $G$ is bipartite, $G$ does not contain $d$ admissible cycles.

By Lemma \ref{usingK_4-}, $G$ does not contain a $K_4^-$ subgraph.
By Lemma \ref{K_3noK_4-}, $G$ does not contain a $K_3$ subgraph.
By Lemma \ref{3connnb}, $G$ is bipartite.
Since $G$ is bipartite, by Lemma \ref{deanc3}, $G$ does not contain a 4-cycle.
By Lemma \ref{deanbi}, $G$ contains $d$ admissible cycles.
But $G$ is bipartite, a contradiction.
\end{proof}

\end{document}